\documentclass[12pt]{amsart}
\usepackage[english]{babel}

\usepackage{amsmath,amssymb,amsbsy,amsfonts,amsthm,latexsym,
  amsopn,amstext,amsxtra,euscript,amscd,stmaryrd,mathrsfs,
  cite,array,mathtools,enumerate}

\usepackage{url}
\usepackage[colorlinks,linkcolor=blue,anchorcolor=blue,citecolor=blue,backref=page]{hyperref}

\usepackage{color}
\usepackage{cancel}

\usepackage[norefs,nocites]{refcheck}

\def\le{\leqslant}
\def\ge{\geqslant}
\def\leq{\leqslant}
\def\geq{\geqslant}
\usepackage{mathtools}
\usepackage{todonotes}

\newcommand\Var{\operatorname{\mathsf{V}}}

\newcommand\Li{\operatorname{Li}}

\usepackage{float}

\hypersetup{breaklinks=true}

\usepackage{enumitem}
\newtheorem{theorem}{Theorem}
\newtheorem{lemma}[theorem]{Lemma}

\newtheorem{cor}[theorem]{Corollary}

\newtheorem{defin}[theorem]{Definition}

\newtheorem{conjecture}[theorem]{Conjecture}

\numberwithin{equation}{section}
\numberwithin{theorem}{section}

\newcommand{\abs}[1]{\left| #1\right|}
\newcommand{\conjugate}[1]{\overline{#1}}
\newcommand{\mirror}[1]{\overleftarrow{#1}}

\newcommand{\dv}{\mid}
\newcommand{\notdv}{\nmid}

\newcommand{\bfell}{{\boldsymbol{\ell}}}

\newcommand{\R}{\mathbb{R}}

\def\\{\cr}
\def\({\left(}
\def\){\right)}
\def\[{\left[}
\def\]{\right]}
\def\<{\langle}
\def\>{\rangle}

\def\le{\leqslant}
\def\ge{\geqslant}
\def\eps{\varepsilon}
\def\mand{\qquad\mbox{and}\qquad}

\newcommand{\pr}[1]{\left( #1\right)}

\newcommand{\norm}[1]{\left\| #1 \right\|}

\newcommand{\e}[1]{\operatorname{e}\pr{ #1}}

\newcommand{\SQF}{Q}

\newcommand\fD{\mathfrak{D}}
\newcommand\fH{\mathfrak{H}}

\newcommand{\p}{\mathfrak{p}}

\def\cB{{\mathcal B}}

\def\cD{{\mathcal D}}

\def\cS{{\mathcal S}}

\def\Z{\mathbb{Z}}
\def\N{\mathbb{N}}

\def\Set{\cB_n}

\newcommand{\tautilde}{\widetilde{\tau}}

\def\vec#1{\mathbf{#1}}

\def\vd{\vec{d}}

\makeatletter
\@namedef{subjclassname@2020}{\textup{2020} Mathematics Subject Classification}
\makeatother

\begin{document}

\title{Reversible primes} 
\author[\tiny{C. Dartyge}]{C\'ecile Dartyge}
\author[\tiny{B. Martin}]{Bruno Martin}
\author[\tiny{J. Rivat}]{Jo\"el Rivat}
\author[\tiny{I. Shparlinski}]{Igor E. Shparlinski}
\author[\tiny{C. Swaenepoel}]{Cathy Swaenepoel}

\address{
  C\'ecile Dartyge,
  Institut Élie Cartan, Universit\'e de Lorraine,
  BP 70239, 54506 Vandœuvre-lès-Nancy Cedex, France
}
\email{cecile.dartyge@univ-lorraine.fr}

\address{
Bruno Martin,
Univ. Littoral C\^ote d'Opale, 
EA 2797 -- LMPA -- Laboratoire de Math\'ematiques pures et appliqu\'ees Joseph Liouville,
F-62228 Calais, 
France. 
}
\email{Bruno.Martin@univ-littoral.fr}

\address{
  Jo\"el Rivat,
  Universit\'e d'Aix-Marseille, Institut Universitaire de France,
  Institut de Math\'ematiques de Marseille CNRS UMR 7373,
  163 avenue de Luminy, Case 907, 13288 Marseille Cedex 9, France.
}
\email{joel.rivat@univ-amu.fr}

\address{
  Igor E. Shparlinski,
  School of Mathematics and Statistics,
  University of New South Wales, Sydney, NSW, 2052, Australia
}
\email{igor.shparlinski@unsw.edu.au}

\address{
  Cathy Swaenepoel,
  Universit\'e Paris Cit\'e and Sorbonne Universit\'e,
  CNRS, IMJ-PRG, F-75013 Paris, France.
}
\email{cathy.swaenepoel@u-paris.fr}

\begin{abstract}
  For an $n$-bit positive integer $a$ written in binary as
$$
a = \sum_{j=0}^{n-1} \varepsilon_{j}(a) \,2^j
$$
where, $\varepsilon_j(a) \in \{0,1\}$, $j\in\{0, \ldots, n-1\}$,
$\varepsilon_{n-1}(a)=1$, let us define
$$
\mirror{a} = \sum_{j=0}^{n-1} \eps_j(a)\,2^{n-1-j},
$$  
the digital reversal of $a$. Also let
$\Set = \{2^{n-1}\leq a<2^n:~a \text{ odd}\}.$ With a sieve
argument, we obtain an upper bound of the expected order of magnitude
for the number of $p \in \Set$ such that $p$ and $\mirror{p}$ are
prime.  We also prove that for sufficiently large $n$,
$$
\abs{\{a \in \Set:~ \max \{\Omega (a), \Omega (\mirror{a})\}\le 8 \}}
\ge c\, \frac{2^n}{n^2},
$$
where $\Omega(n)$ denotes the number of prime factors counted with
multiplicity of $n$ and $c > 0$ is an absolute constant. Finally, we
provide an asymptotic formula for the number of $n$-bit integers $a$
such that $a$ and $\mirror{a}$ are both squarefree.  Our method leads
us to provide various estimates for the exponential sum
$$
\sum_{a \in \Set} \exp\(2\pi i (\alpha a + \vartheta \mirror{a})\)
\quad(\alpha,\vartheta \in\mathbb{R}).   
$$   
\end{abstract}

\keywords{prime numbers, squarefree numbers, reversed binary
  expansions, digits} \subjclass[2020]{11A63, 11N05, 11N36}

\maketitle

\tableofcontents 

\section*{Notations}

Throughout the paper, the notations $U = O(V)$, $U\ll V$ and $V\gg U$
all mean that there is an absolute constant $C>0$ such that
$\abs{U} \leq C V$. If the implicit constant $C$ is allowed to depend
on a parameter $\alpha$ then this dependence is indicated by writing
$U = O_{\alpha}(V)$, $U\ll_{\alpha} V$ or $V \gg_{\alpha} U$.  We also
write $U\asymp V$ if $U \ll V \ll U$ and similarly for
$U\asymp_{\alpha} V$.

For a real number $A> 0$, we write $a \sim A$ to denote $a \in [A,2A)$. 

For a finite set $\cS$ we use $\abs{\cS}$ to denote its cardinality.   

We use $\mu(d)$, $\tau(d)$, $\omega(d)$ and $\Omega (d)$ to denote the
M{\"o}bius function, number of positive divisors, number of distinct
prime factors and the number of prime factors counted with
multiplicity of an integer $d\ge 1$.  We denote by $P^{-}(d)$ and
$P^{+}(d)$ the smallest and the largest prime factors of an integer
$d\ge 2$, respectively.

For a real number $x$ we also set
$$
\e{x} = \exp(2\pi i x) \mand \norm{x} = \min\{|x-k|:~k \in \Z\}. 
$$

For a certain property $\mathsf{P}$, we define
$\mathbf{1}_{\mathsf{P}}$ by $\mathbf{1}_{\mathsf{P}}=1$ if
$\mathsf{P}$ is satisfied and $0$ otherwise.

The letter $p$, with or without indices, always denotes a prime
number.

\section{Introduction}

\subsection{Motivation and set-up} 

Since recently, a large body of research has appeared on arithmetic
properties of integers with various digits restrictions in a given
integer base. For example, this includes the work of Mauduit and
Rivat~\cite{MaRi2} on the sum of digits of primes, the work of
Bourgain~\cite{Bou1, Bou2} and Swaenepoel~\cite{Swa20} on primes with
prescribed digits on a positive proportion of positions in their
digital expansion, and the results of Maynard~\cite{May19,May22} on
primes with missing digits, see
also~\cite{Bug1,BK,col-2009-ellipsephic, DartygeMauduit2000,
  DartygeMauduit2001, DES,DMR,Kar22,Naslund, Pratt} and references
therein for a series of other results about primes and other
interesting integers with various digit restrictions.  In this
direction, polynomial values with digital restrictions have been
studied by Mauduit and Rivat~\cite{MaRi1}, Maynard~\cite{May22} and
very recently by Spiegelhofer~\cite{Spiegelhofer}, see
also~\cite{DMR2011, Stoll, DartygeTenenbaum2006}.

In the present paper, we are interested in a question, which
apparently has never been studied theoretically.  Let $b\ge 2$ an
integer.  For a positive integer~$k$ written in a base $b$ as
$$
  k = \sum_{j=0}^{n-1} \varepsilon_{j}(k) \,b^j
$$
where, $\varepsilon_j(k) \in \{0,\ldots,b-1\}$ for
$j\in\{0, \ldots, n-1\}$ and $\varepsilon_{n-1}(k)\neq 0$, we define
$$
\mirror{k} = \sum_{j=0}^{n-1} \eps_j(k)\,b^{n-1-j}
$$   
as the ``reverse'' of $k$ in base $b$ (throughout the paper, we make
sure that there is no ambiguity on the base $b$).
It is certainly interesting to
understand whether there is any correlation between arithmetic
properties of $k$ and $\mirror{k}$. For instance, a natural question
would be to evaluate the number of $n$-digits integers~$a$ such that
$a$ and $\mirror{a}$ belong to a given set $\mathcal{S}$ defined by a
multiplicative property.

We are especially interested in primality of both $k$ and
$\mirror{k}$.  The prime numbers $p$ such that $\mirror{p}$ is also a
prime number are called {\it reversible primes\/}. Sometimes they are
also referred as ``emirps'', ``reversal primes'' or ``mirror
primes''. The first reversible primes in bases 2 and 10 can be found
in~\cite[\href{https://oeis.org/A074832}{A074832}]{OEIS}
and~\cite[\href{https://oeis.org/A007500}{A007500}]{OEIS},
respectively.

Remarkable examples of reversible primes are of course palindromic
primes, that is, primes $p$ such that $p=\mirror{p}$. Unlike
reversible primes, the distribution of palindromic primes has already
been deeply investigated. Let us denote by $\mathcal{P}_b(x)$ the set
of palindromes less than $x$ in base $b$.  Improving on results by
Banks, Hart and Sakata~\cite{Banks-Hart-Sakata},
Col~\cite{col-2009-palindromes} has obtained an upper bound of the
right order of magnitude for the number of palindromic primes in every
base $b\ge 2$: for $x\ge 2$, we have
$$
\abs{ \{p \in  \mathcal{P}_b(x)  \} }
\ll_b \frac{\abs{ \mathcal{P}_b(x)}}{\log x}. 
$$
Col~\cite{col-2009-palindromes} has also proved that for all $b\ge 2$,
there exists $\kappa_b\ge 1$ such that for a sufficiently large $x$
(depending only on $b$)
$$
\abs{ \{k \in  \mathcal{P}_b(x):~ \Omega(k)\le \kappa_b  \} }
\gg_b \frac{\abs{ \mathcal{P}_b(x)}}{\log x}
$$
and he computed some admissible values of $\kappa_b$.  In particular,
he showed that there are infinitely many binary palindromes $k$ such
that $\Omega(k)\le 60$.  We also mention that Irving~\cite{Irving} has
proved that, for sufficiently large $b$, there exists a 3-digits
palindrome in base $b$ with exactly $2$ prime factors and Banks and
Shparlinski~\cite{Banks-Shparlinski} showed that in any base $b\ge 2$,
for sufficiently large $n$, there exists a $n$-digit palindrome $k$
such that $\omega(k) \ge (\log \log k)^{1+o(1)}$ and a $n$-digit
palindrome $m$ such that $P^{+}(m)\ge (\log m)^{2+o(1)}.$

In this paper, in order to emphasize our main ideas to handle
reversible primes, we choose to concentrate on the emblematic case of
binary expansions. So from now on, we consider $n$-bit integers $k$
such that
$$
k = \sum_{j=0}^{n-1} \varepsilon_{j}(k) \,2^j,
\quad
\mirror{k} = \sum_{j=0}^{n-1} \eps_j(k)\,2^{n-1-j}
$$
with $\varepsilon_j(k) \in \{0,1\}$, $j\in\{0, \ldots, n-1\}$,
$ \varepsilon_{n-1}(k) = 1$.  The sieve method developed by Col in
\cite{col-2009-palindromes} is not suitable to study reversible
primes.  Instead, we develop a two-dimensional sieve approach that
enables us to obtain an upper bound of the expected order of magnitude
for the number of reversible primes and to prove that there are
infinitely many reversible almost primes.  Furthermore, we are able to
get an asymptotic formula for the number of reversible squarefree
integers.

\subsection{Main results}

\subsubsection{Reversible primes}

The $n$-bit prime numbers $p$ such that $\mirror{p}$ is also prime
must satisfy $\eps_0(p)=\eps_{n-1}(p)=1$, which implies that
$p\in \Set$ where
$$
\Set  = \{2^{n-1}\leq a<2^n:~a \text{ odd}\}
$$
is the set of $n$-bit odd integers. Clearly
$$
\abs{\Set} = 2^{n-2}.
$$
We also note that if $a \in \Set$, then
\begin{equation}
  \label{congr_mod_3}
  \mirror{a} \equiv (-1)^{n-1} a \bmod 3,
\end{equation}
so that $3\mid a$ if and
only if $3 \mid \mirror{a}$.

We denote by $\varTheta(n)$ the number of $n$-bit reversible primes:
$$
\varTheta(n) = \abs{\{ 2^{n-1} \leq p < 2^n : ~p \text{ and } \mirror{p} \text{
    are prime} \}}.
$$
It is certainly natural to expect that   
\begin{equation}
  \label{eq:asymp-Theta-c}
  \varTheta(n) = \(c+o(1)\) \frac{2^n}{n^2} \qquad (n \to \infty),
\end{equation}
for some absolute constant $c > 0$. In Section~\ref{sec_count_emirps},
we present numerical investigations and a heuristic argument that
permit us to formulate a conjecture regarding the value of $c$
in~\eqref{eq:asymp-Theta-c}.  We are not able to obtain such an
asymptotic formula but we obtain an upper bound on $\varTheta(n)$ of
the expected order of magnitude. To achieve this, we use a sieve
method based on the following trivial inequality. For any real number
$z \leq 2^{n-1}$ we have
\begin{equation}\label{eq:Theta Thetaz}
  \varTheta(n) \le   \varTheta(n,z),
\end{equation}  
where
$$
  \varTheta(n,z) = |\{a \in \Set:~p \dv a\mirror{a} \Rightarrow p\geq
z\}|.
$$

We use the two-dimensional combinatorial sieve described
in~\cite[p.~308-310]{diamond_halberstam_richert_1988}, with the
associated constant
$$
\beta_2 = 4.2664...
$$ 
from \cite[Appendix~III]{diamond_halberstam_richert_1988}, to
establish the following matching upper and lower bounds on
$\varTheta(n,z)$.

\begin{theorem}\label{thm:Omega-z}
  Let $0<\gamma < 1/(2\beta_2)$. There exists $n_0 \geq 1$, which
  depends only on $\gamma$, such that for $n\geq n_0$, we have
  $$
  \varTheta\left(n,2^{\gamma n} \right)
  \asymp_{\gamma} \frac{2^{n}}{n^2}.
  $$ 
\end{theorem}

Hence, we immediately derive from~\eqref{eq:Theta Thetaz} and
Theorem~\ref{thm:Omega-z} an upper bound on $\varTheta(n)$ of the
expected order of magnitude.

\begin{cor}
  For any integer $n\geq 1$, we have
  $$
    \varTheta(n)  \ll  \frac{2^{n}}{n^2}.
  $$
\end{cor}

Another direct consequence of Theorem~\ref{thm:Omega-z} is the
existence of infinitely many almost prime numbers whose reverse is
also almost prime.  If a $n$-bit integer $a$ is such that all prime
factors of $a\mirror{a}$ are bigger than $2^{\gamma n}$ then
$\max\{\Omega (a), \Omega ( \mirror{a})\}\le \lfloor
1/\gamma\rfloor$. Our limit for the choice of $1/\gamma$ is
$2\beta_2=8.53...$.

\begin{cor}\label{cor:almostprimes}
  There exists $n_0\geq 1$ such that for any integer $n\geq n_0$, we
  have
  $$
  |\{ a  \in \Set : \max \{\Omega (a), \Omega (\mirror{a})\}\le 8\}|
  \gg
    \frac{2^n}{n^2}.
  $$
\end{cor}
The lower bound in Corollary~\ref{cor:almostprimes} is not of the
expected order of magnitude. A power of $\log n$ is missing. This is
due to the fact that the almost primes $a$ and $\mirror{a}$ detected
in Corollary~\ref{cor:almostprimes} are without small prime factors.

\subsubsection{Reversible squarefree integers}

It is also interesting to consider simultaneously squarefree values of
$a$ and $\mirror{a}$.  Thus, we define
$$
\SQF (n)=\abs{\{ a\in  \Set :~\mu ^2(a)=\mu^2(\mirror{a})=1\}}
$$ 
as the cardinality of the set of the $a\in \Set$ such that $a$ and its
reverse $\mirror{a}$ are both squarefree.  This is related to the
sequence~\cite[\href{https://oeis.org/A077337}{A077337}]{OEIS} (in
base $b=10$).  In this case, we are able to obtain an asymptotic
formula for $\SQF(n)$ which matches the following heuristic.

If we choose $a\in \Set $ at random then the probability that
$9 \notdv a$ and $9 \notdv \mirror{a}$ is
\begin{align*}
  \mathbb{P}(9 \notdv a \mbox{ and } 9 \notdv \mirror{a})
  &= 1 - \mathbb{P}(9 \dv a \mbox{ or } 9 \dv \mirror{a})\\
  &= 1 - \mathbb{P}(9 \dv a) - \mathbb{P}(9 \dv \mirror{a})
    + \mathbb{P}(9 \dv a \mbox{ and } 9 \dv \mirror{a}).
\end{align*}
If $3 \dv a$ then $3 \dv \mirror{a}$. Therefore, if $9 \dv a$ then
$\mirror{a}\equiv 0, 3$ or $6 \bmod 9$. Hence we may expect that
$$
  \mathbb{P}(9 \dv a \mbox{ and } 9 \dv \mirror{a}) \approx
  \frac{1}{9} \cdot \frac{1}{3} = \frac{1}{27}
$$
and
$$
  \mathbb{P}(9 \notdv a \mbox{ and } 9 \notdv \mirror{a}) \approx 1-
  \frac{1}{9} -\frac{1}{9} + \frac{1}{27} = \frac{22}{27}.
$$
Moreover, if $a \in \Set $ then $a$ and $\mirror{a}$ are both odd.
Thus we may expect that
\begin{align*}
  \frac{1}{|\Set|} \SQF(n)
  &\approx \mathbb{P}(9 \notdv a \mbox{ and } 9 \notdv \mirror{a})
    \prod_{p\geq 5} \mathbb{P}(p^2 \notdv a \mbox{ and } p^2 \notdv \mirror{a})\\
  &\approx \frac{22}{27}\, \prod_{p\geq 5} \left(1-\frac{1}{p^2}\right)^2
    =  \frac{22}{27}\,\left(\frac{4}{3}\right)^2
    \left(\frac{9}{8}\right)^2 \frac{1}{\zeta(2)^2}\\
  &  = \frac{11}{6}\,\frac{1}{\zeta(2)^2}
    = \frac{66}{\pi^4}.
\end{align*}
Therefore we expect that
$$
  \SQF(n)= |\Set | \left(\frac{66}{\pi^4} + o(1)\right),
$$
which we prove in a quantitative way.

We also define 
$$
 \widetilde \SQF(n) = \abs{\{2^{n-1}\leq a < 2^n :~\mu^2(a)=\mu^2(\mirror{a})=1\}}.
$$

\begin{theorem}\label{thm-squarefree-v2}
  There exists an absolute constant $c>0$
  such that for any $n\geq 1$, we have
  \begin{equation}
    \label{asymp_squarefree_Bn}
    \SQF(n)= |\Set | \left(\frac{66}{\pi^4} + O(\exp (-c\sqrt{n}))\right)
  \end{equation}
  and
  \begin{equation}
    \begin{split} 
    \label{asymp_squarefree}
  \widetilde \SQF(n)  =2^{n-1}\left(\frac{99}{2\pi^4} + O(\exp (-c\sqrt{n}))\right). 
        \end{split}
  \end{equation}
\end{theorem}

Examiming the proof of Theorem~\ref{thm-squarefree-v2} one can easily
see that the value $c = 0.0439$ is admissible.

To conclude this section, we point to the reader that so far, it is
not known whether there exists infinitely many squarefree palindromes.

\subsection{Our approach to estimate $\varTheta(n,z)$}
\label{sec_approach}

As already mentioned, we rely on a classical sieve method.  For
$d\geq 1$, we put
\begin{equation}\label{eq:def Td}
  T_n(d) = |\{ a\in  \Set:~d \dv a  \mirror{a}\}|.
\end{equation}
Clearly $T_n(2)=0$.  For $a\in \Set$ randomly chosen, since
$3 \dv \mirror{a}$ is equivalent to $3 \dv a$ (by
\eqref{congr_mod_3}), the probability that $3 \dv a \mirror{a}$ is
close to $1/3$ and for any prime $p\geq 5$, the events ``$p\dv a$'' and
``$p\dv \mirror{a}$'' should be independent so that the probability
that $p \dv a \mirror{a}$ is expected to be
$\frac{1}{p} + \frac{1}{p} - \frac{1}{p^2} = \frac{2p-1}{p^2}$.
Moreover for distinct prime numbers $p_1$ and $p_2$, the events
``$p_1 \dv a \mirror{a}$'' and ``$p_2 \dv a \mirror{a}$'' are expected
to be independent.  These heuristics lead us to define $R_n(d)$ for
$d$ squarefree by
\begin{equation}\label{eq: def Rd}
  T_n(d) = \frac{f(d)}{d}\, | \Set| + R_n(d),
\end{equation} 
where $f$ is the multiplicative function defined for any prime number
$p$ by
\begin{equation}\label{eq:def f}
  f(p)=
  \left\{
    \begin{array}{ll}
      0 & \mbox{ if } p=2,\\
      1 & \mbox{ if } p=3,\\
      \frac{2p-1}{p} & \mbox{ if } p \geq 5,
    \end{array}
  \right.
\end{equation} 
and $f(p^\nu)=0$ for any $\nu\geq 2$ (it follows that $f(d)\neq 0$ if
and only if $d$ is odd and squarefree).

Let
$$
  V(w)=\prod_{2\le p<w}\left(1-\frac{f(p)}{p}\right).
$$
If $w\leq 3$ then
$$
  V(w) = 1
$$
and if $w>3$ then
$$
  V(w) = \frac{2}{3} \prod_{3<p<w}\left(1-\frac{2p-1}{p^2}\right)
  = \frac{2}{3} \prod_{3<p<w}\left(1-\frac{1}{p}\right)^2  .
$$
The {\it Mertens formula\/}, see, for example,~\cite[Part~I,
Theorem~1.12]{Ten15}, implies that
$$
  V(w)
  \asymp \frac{1}{(\log w)^2}
$$
and that there exists an absolute constant $C>0$ such that for any
$2\leq w_1 \leq w$, we have
$$
  \frac{V(w_1)}{V(w)} \leq \left(\frac{\log w}{\log w_1}\right)^2
  \left(1+\frac{C}{\log w_1}\right).
$$

We are now ready to apply the sieve theorem stated in a more precise
form in~\cite[Equations~(1.17)
and~(1.18)]{diamond_halberstam_richert_1988} with
$$
\mathscr{A} = (a \mirror{a})_{a\in  \Set}, \qquad 
\mathscr{P}=\{p \ge 3:~ p \text{ prime}\}, \qquad \kappa =2.
$$
For any $y\geq z \geq 3$, we have
\begin{equation}\label{eq:Sieve-UB}
  \begin{split} 
    \varTheta(n,z)
    \leq | \Set| V(z)\left(h^+\left(\frac{\log y}{\log z}\right)
      +O\(\frac{\log\log y}{(\log y)^{1/6}}\right)\) \qquad & \\
    + O\Biggl(\sum_{\substack{d \dv P(z)\\d<y}} 
    4^{\omega(d)}|R_n(d)|\Biggr)&,
  \end{split}
\end{equation} 
where 
$$
P(z)=\prod_{3\leq p <z}p,
$$
and $h^+(u)$ is some continuous function, which decreases
monotonically towards 1 as $u\to +\infty$.  Moreover, for any
$y\geq z \geq 3$, we have a similar lower bound
\begin{equation}\label{eq:Sieve-LB}
  \begin{split} 
    \varTheta(n,z)
    \geq | \Set| V(z)\left( h^-\left(\frac{\log y}{\log z}\right) +
      O\left(\frac{\log\log y}{(\log y)^{1/6}}\right)
    \right) \qquad & \\
    + O\Biggl(\sum_{\substack{d \dv P(z)\\d<y}}
    4^{\omega(d)}|R_n(d)|\Biggr)&,
\end{split} 
\end{equation}  
where $h^-(u)$ is some continuous function, which increases
monotonically towards 1 as $u\to +\infty$ and $h^-(u)>0$ for
$u>\beta_2=4.2664...$.

Our main technical result, which we establish in
Section~\ref{sec:proof_lemma_1}, is the following.

\begin{lemma}\label{lemma:sieve-error-term}
  Let $0<\xi<1/2$. There exists $c>0$, which depends only on $\xi$,
  such that for any $n\geq 1$, we have
  $$
  \sum_{\substack{d < 2^{\xi n}\\d \text{ odd}}}
  \mu^2(d)  4^{\omega(d)}|R_n(d)|
  \ll_{\xi} 2^n \exp(-c \sqrt{n}).
  $$
\end{lemma}

In order to prove Lemma~\ref{lemma:sieve-error-term}, we define for
any real numbers $\alpha$ and~$\vartheta$, the exponential sum
\begin{equation}\label{def_F}
  F_n(\alpha,\vartheta)
  = \frac{1}{| \Set|} \sum_{a\in  \Set}\e{\alpha \mirror{a}-\vartheta a}
\end{equation}
and for any integer $d\geq 1$ and $(h_1,h_2)\in \Z^2$,
\begin{equation}\label{def_H}
  H(d,h_1,h_2) = \sum_{\substack{{0\le u,v< d}\\ {d \dv uv}}}
  \e{ \frac{h_1u+h_2v}{d}}. 
\end{equation}
We show in Section~\ref{sec:prel-study-r_nd} that to prove
Lemma~\ref{lemma:sieve-error-term}, it is enough to evaluate the sum
$$ E= \sum_{\substack{d \leq D\\\mu^2(d)=1\\\gcd(d,6)=1}}
\frac{4^{\omega(d)}}{d^2} \sum_{0< h_1,h_2 <d} \abs{H(d,h_1,h_2)}
\abs{F_n\left(\frac{h_2}{d},-\left(\frac{h_1}{d}+\frac{\ell}{3^j}\right)\right)}.
$$
with $D=2^{\xi n}$, $0<\xi<1/2$, $j\in \{0,1\}$,
$\ell \in\{0,\ldots,3^j-1\}$.

We study $F_n(\alpha,\vartheta)$ in detail in Section~\ref{sec_F}.  We
strongly make use of the fact, typical in this kind of situation, that
$|F_n(\alpha,\vartheta)|$ can be written as a product of
cosines. Also, we combine the large sieve inequality and an extension
of the Sobolev-Gallagher inequality to evaluate some discrete averages
of $F_n$.  Here we make the trivial observations that
\begin{equation}\label{symmetry_F}
  \overline{F_n(\alpha,\vartheta)}
  = F_n(-\alpha,-\vartheta)
  = F_n(\vartheta,\alpha)
\end{equation}
and
\begin{equation}\label{eq:Fn-Trivial}
  \abs{  F_n(\alpha,\vartheta)} \le 1.
\end{equation} 
We note that for rational $\alpha$ and $\vartheta$, the
  exponential sums $F_n(\alpha,\vartheta)$ have also been estimated
  in~\cite{Banks-Saidak-Sakata} via a different approach. However the bounds 
  of~\cite{Banks-Saidak-Sakata} are not sufficient for our purpose.  

The quantity $H(d,h_1,h_2)$ is studied in Section~\ref{sec_H}.

In Section~\ref{sec:proof_lemma_1}, we use our estimates on
$ F_n(\vartheta,\alpha)$ and $H(d,h_1,h_2)$ to obtain the bound
$E\ll_{\xi} \exp(-c\sqrt{n})$ for some constant $c>0$ which depends
only on $\xi$.  We complete the proof of the main results in
Section~\ref{completion_proofs}.

\section{Study of $F_n(\alpha,\vartheta)$}\label{sec_F}

\subsection{Generalized Sobolev--Gallagher inequality}

We recall the following: 

\begin{defin}
  We say that a sequence $(x_1,\ldots,x_N) \in \R^N$ is {\it
    $\delta$-spaced modulo~$1$\/} if $\norm{x_i-x_j}\geq \delta$ for
  $1 \leq i < j \leq N$.
\end{defin}

The Sobolev--Gallagher inequality
(see~\cite[Section~3]{montgomery-1978} for relevant references) is
known for continuously differentiable functions.  We extend it to
functions of bounded variation (see for
instance~\cite[p.~355]{titchmarsh-1939}).
\begin{lemma}
  \label{lem:sobolev-gallagher}
  Let $f: \R \to \R$ be a $1$-periodic continuous function of bounded
  variation, with total variation $\Var_f$ on $[0,1]$. For $\delta>0$
  and any $\delta$-spaced sequence $ (x_1,\ldots,x_N) \in \R^N$, we
  have
  $$
  \sum_{n=1}^N \abs{f(x_n)}
  \leq \frac{1}{\delta} \int_0^1 \abs{f(u)}
  du + \frac12 \Var_f .
  $$
\end{lemma}

\begin{proof}
  We adapt the proof in~\cite{montgomery-1978} with Stieltjes
  integrals.  If $\widetilde x_n$ is the fractional part of
  $x_n-x_1+\frac{\delta}{2}$ and $g(x) = f(x+x_1-\frac{\delta}{2})$
  then $\widetilde x_n \in [\frac{\delta}{2},1-\frac{\delta}{2}]$,
  $g(\widetilde x_n) = f(x_n)$, and $g$ is also a $1$-periodic
  continuous function, with a bounded variation $\Var_g$ on $[0,1]$,
  and $\Var_g = \Var_f$.  It is sufficient to prove the result for
  $g(\widetilde x_n)$.  Hence we may assume from now that
  $x_n \in [\frac{\delta}{2},1-\frac{\delta}{2}]$.  For $x\in[0,1]$ we
  denote by $\Var_f(x)$ the total variation of $f$ on $[0,x]$ (thus
  $\Var_f(1) = \Var_f$).  Since $f$ is continuous, by partial
  summation, for $n\in\{1,\ldots,N\}$ we have
  \begin{align*}
    f(x_n)
    =
    \frac{1}{\delta}
    \int_{x_n - \frac{\delta}{2}}^{x_n + \frac{\delta}{2}} f(u) du &
    +
    \int_{x_n - \frac{\delta}{2}}^{x_n}
    \left( \frac{u-x_n}{\delta} + \frac12\right)
    df(u)\\
   & \quad  \qquad +
    \int_{x_n}^{x_n + \frac{\delta}{2}}
    \left( \frac{u-x_n}{\delta} - \frac12\right)
    df(u). 
  \end{align*}
  For $x_n - \frac{\delta}{2} \leq u \leq x_n$
  we have 
  $$0 \leq \frac{u-x_n}{\delta} + \frac12\leq \frac12
  $$
  and for $x_n \leq u \leq x_n+ \frac{\delta}{2}$
  we have 
  $$-\frac12 \leq \frac{u-x_n}{\delta} - \frac12\leq 0.
  $$
  Hence
  \begin{align*}
    \abs{f(x_n)}
     \leq
    \frac{1}{\delta}
    \int_{x_n - \frac{\delta}{2}}^{x_n +\frac{\delta}{2}} \abs{f(u)} du
    +
    \frac12
    \int_{x_n - \frac{\delta}{2}}^{x_n} 
    d \Var_f(u)
    +
    \frac12
    \int_{x_n}^{x_n +\frac{\delta}{2}}
    d \Var_f(u)
    .
  \end{align*}
  Since the intervals $(x_n - \frac{\delta}{2}, x_n +\frac{\delta}{2})$
  are non-overlapping modulo $1$, it follows that
  \begin{align*}
    \sum_{n=1}^N \abs{f(x_n)} &
    \leq
    \frac{1}{\delta}
    \int_0^1 \abs{f(u)} du
    +
    \frac12 \int_0^1 d \Var_f(u)\\
    &\quad  =
    \frac{1}{\delta}
    \int_0^1 \abs{f(u)} du
    +
    \frac12
    \Var_f
    ,
  \end{align*}
  as desired. 
\end{proof}

\subsection{Product formula}

We need the following useful identities for $F_n(\alpha,\vartheta)$
defined by~\eqref{def_F}.

\begin{lemma}\label{FT_product}
  For $(\alpha, \vartheta) \in \R^2$ and $n\geq 3$, we have
  $$
    |F_n(\alpha,\vartheta)| = |F_n(\vartheta,\alpha)|
    =\prod_{j=1}^{n-2} |U\big ( \alpha 2^{n-1-j}-\vartheta 2^j\big )|
  $$
  where
  $$
    U(x) = \frac{1+\e{x}}{2}
    .
  $$
\end{lemma}
\begin{proof}
  By writing
  \begin{align*}
    a &= 2^{n-1} + a_{n-2}2^{n-2} + \cdots + a_12 + 1, \\
    \mirror{a} &= 1 + a_{n-2}2 +\cdots+ a_1 2^{n-2} + 2^{n-1},
  \end{align*}
  we obtain
  $$
  F_n(\alpha,\vartheta)
  =\e{(\alpha-\vartheta)(2^{n-1}+1)}\prod_{j=1}^{n-2} U\big ( \alpha
  2^{n-1-j}-\vartheta 2^j\big ).
  $$
  It is also obvious that
  $F_n(\alpha,\vartheta) = \overline{F_n(\vartheta,\alpha)}$.
\end{proof}

\begin{lemma}\label{FT-splitting-product}
  For $(\alpha, \vartheta) \in \R^2$ and $3 \leq m \leq n-1$, we have
  $$
    |F_n(\alpha,\vartheta)| 
    =
    |F_{m}(\alpha 2^{n-m},\vartheta)|
    \cdot
    |F_{n-m+2}(\alpha,\vartheta 2^{m-2})|
    .
  $$
\end{lemma}
\begin{proof}
  For $3 \leq m \leq n$, by Lemma~\ref{FT_product},
  with the help of $k=n-1-j$ we can write
  $$
    |F_n(\alpha,\vartheta)| 
    =
    \left(
      \prod_{j=1}^{m-2} \abs{U\left( \alpha 2^{n-1-j}-\vartheta 2^j\right)}
    \right)
    \left(
      \prod_{k=1}^{n-m} \abs{U\left( \alpha 2^{k}-\vartheta 2^{n-1-k}\right)}
    \right),
  $$
  while
  \begin{align*}
 |F_m(\alpha 2^{n-m},\vartheta)| 
  & =
    \prod_{j=1}^{m-2} \abs{U\left( \alpha 2^{n-m} 2^{m-1-j}-\vartheta 2^j\right)}\\
  &  =
    \prod_{j=1}^{m-2} \abs{U\left( \alpha 2^{n-1-j}-\vartheta 2^j\right)}
  \end{align*}  
  and
  \begin{align*}
    |F_{n-m+2}(\vartheta 2^{m-2},\alpha)| 
    & =
    \prod_{k=1}^{n-m+2-2}
    \abs{U\left( \vartheta 2^{m-2} 2^{n-m+2-1-k}- \alpha 2^k\right)}\\
    & =
    \prod_{k=1}^{n-m} \abs{U\left( \vartheta 2^{n-1-k}- \alpha 2^k\right)}.
  \end{align*}  
  Since
  $$
    |F_{n-m+2}(\alpha,\vartheta 2^{m-2})|
    = |F_{n-m+2}(\vartheta 2^{m-2},\alpha)| 
  $$
  and $\abs{U(\cdot)}$ is even,
  the result follows.
\end{proof}

\subsection{Individual bounds for $F_n$}

We use an idea of Col~\cite[Proof of Lemme~2]{col-2009-palindromes}
and sharpen the arguments of~\cite[Section~4.2]{col-2009-palindromes}
to get an upper bound for $|F_n(\alpha,\vartheta)|$.

\begin{lemma}\label{lem_maj_U}
  For $(x,y,z)\in\R^3$ we have
  $$
    \abs{U(x)U(y)U(z)}
    \leq
    \frac{1}{4}\abs{U(x+y+z)}+\frac{3}{4}.
  $$
\end{lemma}

\begin{proof}
  For $(x,y,z)\in\R^3$ we have
  \begin{align*}
    \abs{U(x)U(y)U(z)}
    &= \frac{1}{8}
     \bigl|1+\e{x}+\e{y}+\e{z}+\e{x+y}\\
     & \qquad\qquad \qquad +\e{x+z}+\e{y+z}+\e{x+y+z}\bigr |
    \\
    &\leq
      \frac{|1+\e {x+y+z}| + 6}{8}
      =
      \frac{1}{4}\abs{U(x+y+z)}+\frac{3}{4},
  \end{align*}
  which concludes the proof. 
\end{proof}

For any integer $N\geq 1$ and any $\vartheta\in \R$, we define
\begin{equation}\label{eq:definition-G_N}
  G_N(\vartheta)=
  \prod_{j=0}^{N-1}
  \left(
    \frac{1}{4}\abs{U\left(\vartheta 2^j\right)} + \frac{3}{4}
  \right).
\end{equation}

\begin{lemma}\label{lem_maj_F_U}
  For any integer $n \geq 4$ and any $(\alpha,\vartheta)\in\R^2$, we
  have
  \begin{equation}\label{maj_F_G}
    \abs{F_n(\alpha,\vartheta)}
    \leq
    \prod_{j=1}^{n-3}\left(\frac{1}{4}\abs{U\left(3 \vartheta 2^j\right)}
      +\frac{3}{4}\right)^{1/3}
    =G_{n-3}^{1/3}(6\vartheta).
  \end{equation}
\end{lemma}

\begin{proof}
  Applying Lemma~\ref{FT_product} and writing
  \begin{align*}
    \abs{F_n(\alpha,\vartheta)}
    & =
      \prod_{j=1}^{n-2}
      \abs{U\left( \alpha 2^{n-1-j}-\vartheta 2^j\right)}\\
    & \leq
      \prod_{j=1}^{n-3}
      \abs{U\left( \alpha 2^{n-1-j}-\vartheta 2^j\right)}^{1/3}
      \prod_{j=2}^{n-2}
      \abs{\conjugate{U\left( \alpha 2^{n-1-j}-\vartheta 2^j\right)}}^{2/3}
      ,
  \end{align*}
  we get
  $$
  \abs{F_n(\alpha,\vartheta)}
  \leq \prod_{j=1}^{n-3} \abs{ U\left(
      \alpha 2^{n-1-j}-\vartheta 2^j\right) U^2\left( -\alpha
      2^{n-2-j} + \vartheta 2^{j+1}\right) }^{1/3}.
  $$
  Taking
  $$
  x=\alpha 2^{n-1-j}-\vartheta 2^j
  \mand 
  y=z=-\alpha 2^{n-2-j} + \vartheta 2^{j+1}
  $$
  in Lemma~\ref{lem_maj_U} we obtain the desired estimate.
\end{proof}

\begin{lemma}\label{lemma_large_norm}
  For $q\in \Z$, $q\geq 2$ and $\vartheta \in \R\setminus\Z$, the
  integer
  $$
    j = 
    \left\lfloor \frac{\log\frac{q}{(q+1)\norm{\vartheta}}}{\log q}\right\rfloor
  $$
  satisfies
  $$
    \norm{q^j\vartheta} \geq \frac{1}{q+1}.
  $$
\end{lemma} 

\begin{proof}
  Since $\norm{\vartheta}\le 1/2$, we have $j\ge 0$.  Therefore
  $q^j \in \Z$ and we have
  \begin{math}
    \norm{q^j \vartheta} = \norm{q^j \norm{\vartheta}}
  \end{math}
  by parity and periodicity.  By definition of $j$, we have
  $$
  \frac{1}{q+1} < q^j \norm{\vartheta}
  \leq \frac{q}{q+1} = 1-\frac{1}{q+1}, 
  $$
  which gives the result.
\end{proof}

We are now ready to establish one of our main technical tools.
\begin{lemma}\label{lem_maj_Fn}
  For $\alpha \in \R$ and $n,\ell, h, d \in \Z$ such that $n\geq 4$,
  $d\geq 5$, $d$ is odd and $d \notdv 3h$, we have
  $$
    \left|F_n\left(\alpha,\frac{h}{d}+\frac{\ell}{3}\right)\right|
    \ll \exp\left(\frac{-c_0\,n}{\log\left(\frac{4d}{3}\right)}\right)
    $$
  where
  $$
    c_0 = \frac{1}{3} \log\left(\frac{8}{7}\right)\log 2 = 0.0308\ldots.
  $$
\end{lemma}
\begin{proof}
  Since $\abs{U(x)} = \cos\pi\norm{x}$ for any $x\in\R$, we have by
  Lemma~\ref{lem_maj_F_U},
  \begin{align*}
    \left|F_n\left(\alpha,\frac{h}{d}+\frac{\ell}{3}\right)\right|
    &\leq
      \prod_{j=1}^{n-3}\left(\frac{1}{4}\cos \pi \norm{3 \left(\frac{h}{d}
      +\frac{\ell}{3}\right) 2^j} +\frac{3}{4}\right)^{1/3}\\
    &\quad= \prod_{j=1}^{n-3}\left(\frac{1}{4}\cos \pi \norm{\frac{3h2^{j}}{d}}
      +\frac{3}{4}\right)^{1/3}.
  \end{align*}
  Thus, defining 
  $$
    J =  1+\left \lfloor\frac{\log\left(\frac{2d}{3}\right)}{\log 2} \right\rfloor
    \geq 1,
    \qquad
    K = \left\lfloor \frac{n-3}{J}\right\rfloor,
    \qquad
    \vartheta_{k} = \frac{3h2^{kJ+1}}{d}
  $$
  for any integer $k$, we may write
  \begin{align*}
    \left|F_n\left(\alpha,\frac{h}{d}+\frac{\ell}{3}\right)\right|
    &\leq
      \prod_{0\leq k < K} \prod_{0\leq j < J}
      \left(\frac{1}{4}\cos \pi \norm{ 2^j\vartheta_{k}}
      +\frac{3}{4}\right)^{1/3}.
  \end{align*}
  We fix $k\in \{0,\ldots,K-1\}$.
  Since $d$ is odd and $d \notdv 3h$, we have
  \begin{math}
    \norm{\vartheta_k} \geq \frac{1}{d}.
  \end{math}
  Thus, denoting
  $$
    j_k = \left\lfloor \frac{\log\frac{2}{3\norm{\vartheta_k}}}{\log 2}\right\rfloor,
  $$
  we have
  $$
    0 \leq j_k \leq J-1
  $$
  and by Lemma~\ref{lemma_large_norm}, 
  $$
    \norm{2^{j_k}\vartheta_k} \geq \frac{1}{3}.
  $$
  It follows that 
  $$
    \left|F_n\left(\alpha,\frac{h}{d}+\frac{\ell}{3}\right)\right|
    \leq \prod_{0\leq k < K}
    \left(\frac{1}{4}\cos \pi \norm{ 2^{j_k}\vartheta_{k}} +\frac{3}{4}\right)^{1/3}
    \leq \left(\frac{7}{8}\right)^{K/3}.
  $$
  Since
  \begin{math}
    J \leq \frac{\log\left(\frac{4d}{3}\right)}{\log 2},
  \end{math}
  we have
  $$
    K > \frac{n-3}{J} - 1
    \geq \frac{(n-3)\log 2}{\log\left(\frac{4d}{3}\right)} - 1
    = \frac{n\log 2}{\log\left(\frac{4d}{3}\right)} +O(1),
  $$
  we get
  \begin{align*}
    \left|F_n\left(\alpha,\frac{h}{d}+\frac{\ell}{3}\right)\right|
    \leq \left(\frac{7}{8}\right)^{\frac{K}{3}}
    \ll \exp\left(\frac{-n \log\left(\frac{8}{7}\right)\log 2}
    {3\log\left(\frac{4d}{3}\right)}\right),
  \end{align*}
  which completes the proof.
\end{proof}

\subsection{Bounds on some continuous averages}

To fix the ideas, it is interesting to note that by orthogonality we
have
\begin{align*}
  &  \int_0^1 \abs{F_n(\alpha,\vartheta)}^2 d\vartheta = \frac{1}{|
    \Set|^2} \int_0^1
    \abs{\sum_{a\in  \Set}\e{\alpha \mirror{a}-\vartheta a}}^2 d\vartheta\\
  &\qquad  = \frac{1}{| \Set|^2} \sum_{a_1\in \Set}\sum_{a_2\in \Set}
    \e{\alpha(\mirror{a_1}-\mirror{a_2})} \int_0^1
    \e{\vartheta(a_2-a_1)} d\vartheta = \frac{1}{| \Set|}, 
\end{align*}
and by the Cauchy--Schwarz inequality this gives the trivial upper
bound
$$
\int_0^1 \abs{F_n(\alpha,\vartheta)} d\vartheta \leq \left( \int_0^1
  \abs{F_n(\alpha,\vartheta)}^2 d\vartheta\right)^{1/2} = \frac{1}{|
  \Set|^{1/2}}.
$$

The following estimate of $G_N$ (defined by~\eqref{eq:definition-G_N})
is a key argument in the sequel.

\begin{lemma}\label{lem:integral-G_N}
  For any integer $N\geq 1$ and any real number
  $\kappa \in \left[0,1\right]$, we have
  $$
    \int_0^1  G_N^{\kappa}(\vartheta) d\vartheta \leq C(\kappa)^N
  $$
  where
  $$
    C(\kappa) = 
    \left( \frac{\sqrt{2}}{8} + \frac{3}{4} \right)^{\kappa}.
  $$
\end{lemma}

\begin{proof} 
  For any $N\geq 1$, we have
  \begin{align*}
    \int_0^1  G_N^{\kappa}(\vartheta) d\vartheta
    &=
      \int_0^1
      \left(
      \frac{1}{4}\abs{U\left(\vartheta \right)} +\frac{3}{4}
      \right)^{\kappa}
      G_{N-1}^{\kappa}(2\vartheta) d\vartheta\\
    &=
      \int_0^{1/2}
      \left(
      \frac{1}{4}\abs{U\left(\vartheta \right)} +\frac{3}{4}
      \right)^{\kappa}
      G_{N-1}^{\kappa}(2\vartheta) d\vartheta\\
      & \qquad \qquad  +
      \int_0^{1/2}
      \left(
      \frac{1}{4}\abs{U\left(\vartheta+\frac{1}{2} \right)} +\frac{3}{4}
      \right)^{\kappa}
      G_{N-1}^{\kappa}(2\vartheta) d\vartheta\\
    &=
      \int_0^{1}
      \left(
      \frac{1}{4}\abs{U\left(\frac{\vartheta}{2} \right)} +\frac{3}{4}
      \right)^{\kappa}
      G_{N-1}^{\kappa}(\vartheta) \frac{d\vartheta}{2}\\
      & \qquad \qquad    + \int_0^{1}
      \left(
      \frac{1}{4} \abs{U\left(\frac{\vartheta+1}{2}\right)} +\frac{3}{4}
      \right)^{\kappa}
      G_{N-1}^{\kappa}(\vartheta) \frac{d\vartheta}{2}\\
    &=
      \int_0^{1} \Phi_\kappa(\vartheta) G_{N-1}^{\kappa}(\vartheta) d\vartheta, 
  \end{align*}
  where
  $$
    \Phi_\kappa(\vartheta)
    =
    \frac{1}{2}
    \left(
      \frac{1}{4} \abs{U\left(\frac{\vartheta}{2}\right)} + \frac{3}{4}
    \right)^{\kappa}
    +
    \frac{1}{2}
    \left(
      \frac{1}{4}\abs{U\left(\frac{\vartheta+1}{2}\right)} + \frac{3}{4}
    \right)^{\kappa}.
  $$
  Since $x\mapsto x^{\kappa}$ is concave for $0\leq\kappa\leq 1$ and
  $\abs{\cos x} + \abs{\sin x} \leq \sqrt{2}$, we have
  \begin{align*}
    \Phi_\kappa(\vartheta)
    &
      \leq
      \left(
      \frac{1}{2}
      \left(
      \frac{1}{4}\abs{U\left(\frac{\vartheta}{2}\right)} +\frac{3}{4}
      \right)
      +
      \frac{1}{2}
      \left(
      \frac{1}{4}\abs{U\left(\frac{\vartheta+1}{2}\right)} +\frac{3}{4}
      \right)
      \right)^{\kappa}\\
    &=
      \left(
      \frac{1}{8}\left(
      \abs{U\left(\frac{\vartheta}{2}\right)}
      + \abs{U\left(\frac{\vartheta+1}{2}\right)}
      \right)
      + \frac{3}{4} \right)^{\kappa}\\
    &=
      \left(
      \frac{1}{8}
      \left(
      \abs{\cos \frac{\pi \vartheta}{2}} + \abs{\sin\frac{\pi \vartheta}{2}}
      \right)
      + \frac{3}{4}
      \right)^{\kappa}\\
    &\leq \left( \frac{\sqrt{2}}{8}
      + \frac{3}{4} \right)^{\kappa} = C(\kappa).
  \end{align*}
  We have proved that
  $$
  \int_0^1  G_N^{\kappa}(\vartheta) d\vartheta
  \leq C(\kappa) \int_0^1  G_{N-1}^{\kappa}(\vartheta) d\vartheta.
  $$
  By induction, it follows that
  $$
    \int_0^1  G_N^{\kappa}(\vartheta) d\vartheta 
    \leq C(\kappa)^N \int_0^1  G_{0}^{\kappa}(\vartheta) d\vartheta 
    =  C(\kappa)^N
  $$
  which completes the proof.
\end{proof}

\begin{lemma}\label{maj_L_1_norm_G'}
  For any integer $N\geq 1$ and any real number $\kappa >0$,
  $G_{N}^{\kappa}$ admits almost everywhere a derivative
  $\left(G_{N}^{\kappa} \right)'$ which satisfies
  $$
    \norm{\left(G_{N}^{\kappa} \right)'}_1
    \leq \frac{\kappa \pi}{3} \, 2^N \norm{G_{N}^{\kappa}}_1.
  $$  
\end{lemma}

\begin{proof}
  Since
  $$
    G_{N}(\vartheta)
    = \prod_{j=0}^{N-1}
    \left(
      \frac{1}{4}\abs{\cos(\pi \vartheta 2^j)} + \frac{3}{4}
    \right),
  $$
  we have for almost all $\vartheta \in \R$,
  $$
    \abs{\left(G_{N}^{\kappa} \right)'(\vartheta)}
    \leq
    \abs{G_{N}^{\kappa}(\vartheta)}
    \sum_{j=0}^{N-1}
    \frac{\kappa \frac{\pi 2^j}{4}\abs{\sin(\pi\vartheta2^j)}}{
      \frac{1}{4}\abs{\cos\left(\pi\vartheta 2^j\right)} + \frac{3}{4}}
    \leq
    \frac{\kappa \pi}{3} 2^N
    \abs{G_{N}^{\kappa}(\vartheta)}
  $$
  and the result follows.
\end{proof}

\subsection{Bounds on some discrete averages of $F_n$}

For a triple  $\cD = (D_1,D_2,D_3)$ we define the set
\begin{align*}
  \fD(\cD) = \bigl\{(d_1,d_2,d_3)\in \N^3:~ d_1 \sim D_1, \ d_2 & \sim D_2, \ d_3\sim D_3, \\
   & \gcd(d_1d_2d_3,6)=1\bigr\}, 
\end{align*}
and for
$$
  \vd  =(d_1,d_2,d_3) \in \fD(\cD)
$$
we define
\begin{align*}
  \fH(\vd) = \bigl\{(h_1,h_2)\in \N^2:~0 <h_1& <d_2d_3, \  0<h_2<d_1d_3,\\
  & \gcd(h_1,d_2d_3)= \gcd (h_2,d_1d_3)=1\bigr\}.
\end{align*}

Finally, for $\cD = (D_1,D_2,D_3)$,
$\bfell = (\ell_1,\ell_2)\in\Z^2$ and $r\in\{1,2\}$, we define
\begin{equation}\label{notation_R_D_l}
\begin{split}
 & M_{r}(n;\cD,\bfell)\\
 &\qquad  =
  \sum_{\vd  \in \fD(\cD)}
  \sum_{(h_1, h_2)\in  \fH(\vd)}
  \abs{
    F_n\left(
      \frac{h_2}{d_1d_3}+\frac{\ell_1}{3},
      -\frac{h_1}{d_2d_3}-\frac{\ell_2}{3}
    \right)
  }^r.
  \end{split}
\end{equation}

Let us introduce the non-decreasing function $\tautilde$ defined by
$$
\tautilde(t)= \max_{d \leq t} \tau(d),
\qquad
t \in [1,+\infty)
$$
and observe that, for any positive integer $d$,
\begin{equation}
  \label{eq:tau tautilde}
  2^{\omega(d)}\leq \tau(d)  \leq  \tautilde(d).
\end{equation}

Recalling the definition of $C(\kappa)$ from
Lemma~\ref{lem:integral-G_N}, we define
$$
  \eta_0 = -\frac{\log C(2/3)}{\log  2}\approx 0.073\ldots.
$$

In order to bound $M_{1}(n;\cD,\bfell)$, we first establish the
following bound on $M_{2}(n;\cD,\bfell)$.
\begin{lemma}\label{lem_maj__R_D_l_squared}
  For any $n\geq 10$, any
  $\cD = (D_1,D_2,D_3) \in \left[1,+\infty\right)^3$ and any
  $\varepsilon \in \left(0,1\right]$ such that  
  \begin{equation}\label{eq:cond_squared_D1D2D3}
      D_1^2D_2^{2\varepsilon}D_3^{1+\varepsilon} \leq 2^{n-10},
  \end{equation}
  we have for any $\bfell=(\ell_1,\ell_2)\in\Z^2$,
  $$
    \frac{M_{2}(n;\cD,\bfell)}{D_1D_2D_3^2}
    \ll \tautilde(4D_2D_3) 
    \frac{D_2}{D_1} \left(D_2^{2} D_3\right)^{-\varepsilon\eta_0}
  $$
  where the implicit constant is absolute.
\end{lemma}

\begin{proof}
  We introduce two integer parameters $n_{1},n_{2} \geq 4$ such that
  $n_{1}+n_{2} \leq n$.  Applying Lemma~\ref{FT-splitting-product}
  twice and recalling~\eqref{eq:Fn-Trivial}, we derive that there
  exists $(u,v)\in\N^2$ such that for any
  $(\alpha,\vartheta) \in \R^2$,
  \begin{equation}\label{eq:FnFn1Fn2}
    \abs{F_n(\alpha,\vartheta)}
    \leq
    \abs{F_{n_1}(\alpha2^{u},\vartheta )}
    \cdot
    \abs{F_{n_2}(\alpha ,\vartheta2^{v})}
    .
  \end{equation}
  Furthermore, by~\eqref{maj_F_G}, we have
  $$
    \abs{F_{n}\(\alpha,\vartheta\)} \leq G_{n_{1}-3}^{1/3}(6\vartheta)
    \cdot \abs{F_{n_{2}}\(\alpha,\vartheta 2^{v}\)} .
  $$
  Moreover, since $G_{n_{1}-3}$ is $1$-periodic and even, we have
  $$
    G_{n_{1}-3}\left(6\left(-\frac{h_1}{d_2d_3}
        -\frac{\ell_2 }{3}\right)\right)
    =  G_{n_{1}-3}\left(\frac{6h_1}{d_2d_3}\right)
    .
  $$  
  Hence,
  \begin{align*}
    M_2(n;\cD,\bfell) \leq \sum_{\vd \in \fD(\cD)} \sum_{(h_1, h_2)\in
      \fH(\vd)} &
    G_{n_{1}-3}^{2/3}\left(\frac{6h_1 }{d_2d_3}\right)\times\\
   & \abs{ F_{n_{2}}\left( \frac{h_2}{d_1d_3}+\frac{\ell_1}{3},
        -\frac{h_1 2^{v}}{d_2d_3}-\frac{\ell_2 2^{v}}{3} \right) }^2.
  \end{align*}
  For given $d_3$ and $\ell_1$, the points
  $$
  \frac{h_2}{d_1d_3}+\frac{\ell_1}{3},
  \quad d_1 \sim D_1,\
  0 < h_2 < d_1d_3,\ 
  \gcd(h_2, d_1d_3) =1,
  $$
  are $(8D_1^2D_3)^{-1}$-spaced modulo~$1$.  By summing over $d_1$ and
  $h_2$, we obtain by the large sieve inequality (see for
  instance~\cite{montgomery-1978}),
  $$
    M_2(n;\cD,\bfell)
    \ll\sum_{\substack{{d_2\sim D_2}\\ {d_3\sim D_3}\\ {\gcd(d_2d_3,6)=1}}}
    \sum_{\substack{ {0<h_1<d_2d_3}\\{\gcd(h_1,d_2d_3)=1}}}
    G_{n_{1}-3}^{2/3} \left(\frac{6h_1}{d_2d_3}\right)
    (D_1^2 D_32^{-n_{2}} + 1).
  $$
  Since $\gcd(d_2d_3,6)=1$, we have by a change of variable,
  $$
    M_2(n;\cD,\bfell)
    \ll
    (D_1^2 D_32^{-n_{2}} + 1)
    \sum_{\substack{{d_2\sim D_2}\\ {d_3\sim D_3}\\ {\gcd(d_2d_3,6)=1}}}
    \sum_{\substack{ {0<h_1<d_2d_3}\\{\gcd(h_1,d_2d_3)=1}}}
    G_{n_{1}-3}^{2/3}\left(\frac{h_1}{d_2d_3}\right)
    .
  $$
  Writing
  $$
    \tau(d;D_2,D_3)
    =
    \abs{\{
      (d_2,d_3):d_2\sim D_2, d_3\sim D_3,
      d=d_2d_3
      \}},
  $$
  we have
  \begin{align*}
    M_2(&n;\cD,\bfell)\\
    &\ll
    \left( D_1^2 D_3 2^{-n_{2}} + 1\right) 
    \sum_{d\in [D_2 D_3,4 D_2 D_3)} 
    \tau(d;D_2,D_3)
    \sum_{\substack{ {0<h_1<d}\\ {\gcd(h_1,d)=1}}}
    G_{n_{1}-3}^{2/3}\left(\frac{h_1}{d} \right)
    ,
  \end{align*}
  and observing that by~\eqref{eq:tau tautilde} we have
  $$
    \tau(d;D_2,D_3)  \leq \tau(d) \leq \tautilde(d) \leq \tautilde(4D_2D_3)  ,
  $$
  we immediately derive
  \begin{align*}
    M_2(&n;\cD,\bfell)\\
    &\ll
    \tautilde(4D_2D_3) 
    \left( D_1^2 D_3 2^{-n_{2}} + 1\right)
    \sum_{d\in [D_2 D_3,4 D_2 D_3)}
    \sum_{\substack{ {0<h_1<d}\\ {\gcd(h_1,d)=1}}}
    G_{n_{1}-3}^{2/3}\left(\frac{h_1}{d} \right)  .
  \end{align*}
  By Lemma~\ref{lem:integral-G_N}, we have
  $$
    \norm{G_{n_{1}-3}^{2/3}}_1 \ll C(2/3)^{n_{1}} = 2^{-n_{1}\eta_0}
  $$
  and by Lemma~\ref{maj_L_1_norm_G'}, the variation of
  $G_{n_{1}-3}^{2/3}$ on $[0,1]$ is
  $$
    \Var_{G_{n_{1}-3}^{2/3}}
    \ll
    2^{n_{1}} \norm{G_{n_{1}-3}^{2/3}}_1.
  $$
  Since the points $h_1/d$ are $(16 D_2^2 D_3^2)^{-1}$-spaced modulo
  $1$, it follows from Lemma~\ref{lem:sobolev-gallagher} that
  $$
    M_2(n;\cD,\bfell)
    \ll 
    \tautilde(4D_2D_3)
    \left( D_1^2 D_3 2^{-n_{2}} + 1\right)
    \left( D_2^2 D_3^2 + 2^{n_{1}}\right)
    2^{-n_{1}\eta_0}
    .
  $$
  We choose for $n_1$ and $n_2$ the unique integers such that 
  $$
    2^{n_1-5} < (D_2^2D_3)^{\varepsilon} \leq 2^{n_1-4},
    \quad
    2^{n_2-5} <  \ D_1^2D_3 \leq 2^{n_2-4}
    .
  $$
  Since $D_1, D_2, D_3 \geq 1$ we have $n_1\geq 4$, $n_2\geq 4$ and
  by~\eqref{eq:cond_squared_D1D2D3} we have $n_1 + n_2 \leq n$. Since
  $\varepsilon \in \left(0,1\right]$, this leads to
  $$
    M_2(n;\cD,\bfell)
    \ll
    \tautilde(4D_2D_3)
    D_2^2 D_3^2
    \left(D_2^2D_3\right)^{-\varepsilon\eta_0}
  $$
  and completes the proof.
\end{proof}

We are now able to bound $M_{1}(n;\cD,\bfell)$.

\begin{lemma}\label{lem:maj_R_D_l}
  For any $n\ge 22$, any
  $\cD = (D_1,D_2,D_3) \in \left[1,+\infty\right)^3$ and any
  $\varepsilon \in \left(0,1\right]$ such that
  \begin{equation}\label{eq:condD1D2D3}
    (D_1D_2D_3)^{2(1+\varepsilon)} \leq 2^{n-22},
  \end{equation}
  we have for any $\bfell=(\ell_1,\ell_2)\in\Z^2$,
  $$
    \frac{M_1(n;\cD,\bfell)}{D_1D_2D_3^2}
    \ll
    \tautilde(4D_1D_3)^{1/2} \tautilde(4D_2D_3)^{1/2}
    \left(D_1D_2D_3\right)^{-\varepsilon \eta_0}
  $$
  where the implicit constant is absolute.
\end{lemma}

\begin{proof}
  Let $n_1,n_2 \geq 10$ such that $n_1+n_2 \leq n$. 
  Similarly to~\eqref{eq:FnFn1Fn2}, we see that there exists
  $(u,v) \in\N^2$ such that for any $(\alpha,\vartheta) \in \R^2$,
  $$
    \abs{F_n(\alpha,\vartheta)}
    \leq
    \abs{F_{n_1}(\alpha2^{u},\vartheta )}
    \cdot
    \abs{F_{n_2}(\alpha ,\vartheta2^{v})}
    .
  $$
  It is convenient to define
  $$
  (u_1,v_1) = (u,0) \mand (u_2,v_2) = (0,v) .
  $$
  Applying the Cauchy--Schwarz inequality, we get
  \begin{align*}
    &  M_1(n;\cD,\bfell)^2\\
    & \qquad     \leq 
      \prod_{k=1}^2
      \sum_{\vd  \in \fD(\cD)}
      \sum_{(h_1, h_2)\in  \fH(\vd)}
      \abs{
      F_{n_k}\left(
      \frac{h_2 2^{u_k}}{d_1d_3}+\frac{\ell_12^{u_k}}{3},
      -\frac{h_1 2^{v_k}}{d_2d_3}-\frac{\ell_2 2^{v_k}}{3}
      \right)
      }^2
      .
  \end{align*}
  The maps $\alpha \mapsto F_{n_k}(\alpha,\vartheta)$ and
  $\vartheta \mapsto F_{n_k}(\alpha,\vartheta)$ are $1$-periodic, and
  since
  $$
    \gcd(2^{u_k},d_1d_3)=\gcd(2^{v_k},d_2d_3)=1,
  $$
  the integer $h_1 2^{v_k}$ runs over all residue classes modulo
  $d_2d_3$ coprime with $d_2d_3$ and $h_2 2^{u_k}$ runs over all
  residue classes modulo $d_1d_3$ coprime with $d_1d_3$. Therefore, by
  a change of variables we obtain
  \begin{align*}
    & M_1(n;\cD,\bfell)^2\\
    &\qquad \leq 
      \prod_{k=1}^2
      \sum_{\vd  \in \fD(\cD)}
      \sum_{(h_1, h_2)\in  \fH(\vd)}
      \abs{
      F_{n_k}\left(
      \frac{h_2}{d_1d_3}+\frac{\ell_12^{u_k}}{3},
      -\frac{h_1}{d_2d_3}-\frac{\ell_2 2^{v_k}}{3}
      \right)
      }^2\\
    &\qquad\quad  = \prod_{k=1}^2
      M_2(n_k;\cD,\bfell_k),
  \end{align*}
  where $\bfell_k = (\ell_12^{u_k},\ell_22^{v_k})$.  We choose for
  $n_1$ and $n_2$ the unique integers such that
  $$
    2^{n_1-11} < D_1^2D_2^{2\varepsilon}D_3^{1+\varepsilon} \leq
    2^{n_1-10}, \quad 2^{n_2-10} <
    D_1^{2\varepsilon}D_2^2D_3^{1+\varepsilon} \leq 2^{n_2-10} .
  $$
  Since $D_1, D_2, D_3 \geq 1$ we have $n_1\geq 10$, $n_2\geq 10$ and
  by~\eqref{eq:condD1D2D3} we have $n_1 + n_2 \leq n$.  By applying
  Lemma~\ref{lem_maj__R_D_l_squared} with $n$ replaced by $n_1$, we
  obtain
  $$
    \frac{M_2(n_1;\cD,\bfell_1)}{D_1D_2D_3^2}
    \ll
    \tautilde(4D_2D_3) 
    \frac{D_2}{D_1} \left(D_2^{2} D_3\right)^{-\varepsilon\eta_0}.
  $$
  To bound $M_2(n_2;\cD,\bfell_2)$, we first note that
  by~\eqref{symmetry_F}, we have
  $$
    M_2(n_2;\cD,\bfell_2) = M_2(n_2;\widetilde \cD, \widetilde \bfell_2), 
  $$
  where $\widetilde \cD=(D_2,D_1,D_3)$ and
  $\widetilde \bfell_2=(\ell_22^{v_2},\ell_12^{u_2})$. Next, by applying
  Lemma~\ref{lem_maj__R_D_l_squared} with $n$, $\cD$ and $\bfell$
  replaced by $n_2$, $\widetilde \cD$ and $\widetilde \bfell_2$, we get
  $$
    \frac{M_2(n_2;\cD,\bfell_2)}{D_1D_2D_3^2}
    \ll
    \tautilde(4D_1D_3) 
    \frac{D_1}{D_2} \left(D_1^{2} D_3\right)^{-\varepsilon\eta_0}.
  $$
  It follows that
  $$
    \frac{M_1(n;\cD,\bfell)^2}{(D_1D_2D_3^2)^2}
    \ll
    \tautilde(4D_1D_3) \tautilde(4D_2D_3)  
    \left(D_1D_2 D_3\right)^{-2\varepsilon\eta_0},
  $$
  as desired.
\end{proof}

\section{Study of $H(d,h_1,h_2)$}\label{sec_H}
We recall that, by~\eqref{def_H}, for any integer $d\geq 1$ and
$(h_1,h_2)\in \Z^2$,
$$
  H(d,h_1,h_2) = \sum_{\substack{{0\le u,v< d}\\ {d \dv uv}}}
  \e{ \frac{h_1u+h_2v}{d}}.
$$

\begin{lemma}\label{H}
  For any $(h_1,h_2)\in \Z^2$ the function $d \mapsto H(d,h_1,h_2)$ is
  multiplicative and for any prime number $p$, we have
  $$
    H(p,h_1,h_2) = p\mathbf{1}_{p \dv h_1}+p\mathbf{1}_{p \dv h_2}-1.
  $$
\end{lemma}

\begin{proof}
  For $d_1\geq 1$ and $d_2\geq 1$ with $\gcd(d_1,d_2)=1$ the summation
  over $u$ with $0\leq u < d=d_1 d_2$ may be replaced by
  $d_2 u_1 + d_1 u_2$ with $0\leq u_1 < d_1$, $0\leq u_2 < d_2$ and
  similarly for $v$. We have
  \begin{align*}
    \frac{h_1 u}{d}
    &\equiv
      \frac{h_1 (d_2 u_1 + d_1 u_2)}{d_1 d_2}
      \equiv
      \frac{h_1 u_1}{d_1} + \frac{h_1 u_2}{d_2}
      \bmod 1,
    \\
    \frac{h_2 v}{d}
    &\equiv
      \frac{h_2 (d_2 v_1 + d_1 v_2)}{d_1 d_2}
      \equiv
      \frac{h_2 v_1}{d_1} + \frac{h_2 v_2}{d_2}
      \bmod 1, 
  \end{align*}
  and also
  \begin{align*}
    u v
    &\equiv
      (d_2 u_1 + d_1 u_2) (d_2 v_1 + d_1 v_2)
      \equiv d_2^2 u_1 v_1 \bmod d_1, 
    \\
    u v
    &\equiv
      (d_2 u_1 + d_1 u_2) (d_2 v_1 + d_1 v_2)
      \equiv d_1^2 u_2 v_2 \bmod d_2
      .
  \end{align*}
  Since $\gcd(d_1,d_2)=1$, the condition $d_1d_2\dv uv$, may be
  replaced by $d_1 \dv u_1 v_1$ and $d_2 \dv u_2 v_2$.  This leads to
  \begin{align*}
    & H(d_1 d_2,h_1,h_2)\\
    & \qquad \qquad =
      \sum_{\substack{0\leq u_1, v_1<d_1\\ d_1 \dv u_1 v_1}} 
    \e{ \frac{h_1 u_1 + h_2 v_1}{d_1}}
    \sum_{\substack{0\leq u_2, v_2<d_2\\ d_2 \dv u_2 v_2}} 
    \e{ \frac{h_1 u_2 + h_2 v_2 }{d_2}},
    \\
    &\qquad \qquad =
      H(d_1,h_1,h_2) \, H(d_2,h_1,h_2)
      ,
  \end{align*}
  which shows that $d \mapsto H(d,h_1,h_2)$ is multiplicative.

  If $d=p$ is a prime number then $p \dv uv$ means $p \dv u$ or
  $p \dv v$, namely $u=0$ or $v=0$, hence
  \begin{align*}
    H(p,h_1,h_2) &
                   =
                   \sum_{0\leq v < p} \e{\frac{h_2 v}{p}}
                   +
                   \sum_{0\leq u < p} \e{\frac{h_1 u}{p}}
                   - 1\\
                 &  =
                   p\mathbf{1}_{p \dv h_1}+p\mathbf{1}_{p \dv h_2}-1, 
  \end{align*}
  as claimed. 
\end{proof}

We now recall the definition~\eqref{eq:def f} of the multiplicative
function $f(d)$.

\begin{lemma}\label{lem_H_f}
  For any squarefree integer $d\geq 1$ such that $\gcd(d,6)=1$, we
  have
  $$
    H(d,0,0)=df(d).
  $$
\end{lemma}

\begin{proof}
  For any prime number $p\geq 5$, we have
  $$
    H(p,0,0)=2p-1 = pf(p), 
  $$
  and the result now follows from the multiplicativity of the
  functions $H(d,0,0)$ and $f(d)$.
\end{proof}

\section{Average order of useful multiplicative functions}

We use the following upper bounds.
\begin{lemma}\label{lem_maj_average_omega}
  For any real numbers $z>0$ and $x\geq 2$, we have
  \begin{equation}
    \label{maj_sum_omega_over_n}
    \sum_{n\leq x} \frac{\mu^2(n)z^{\omega(n)}}{n} \ll_z (\log x)^z
  \end{equation}
  and
  \begin{equation}
    \label{maj_sum_omega}
    \sum_{n\leq x} \mu^2(n)z^{\omega(n)} \ll_z x(\log x)^{z-1}.
  \end{equation}
\end{lemma}
Stronger results may be obtained
by~\cite[Corollary~2.15]{montgomery-vaughan-2007}
and~\cite[Chapter~II, Theorem~6.1]{Ten15}.

\begin{proof}
  Since $n\mapsto \mu^2(n)z^{\omega(n)}$ is multiplicative with
  non negative values, we have
  $$
    \sum_{n\leq x} \frac{\mu^2(n)z^{\omega(n)}}{n}
    \leq
    \prod_{p\leq x} \left(1+\frac{\mu^2(p)z^{\omega(p)}}{p} \right)
    = \prod_{p\leq x} \left(1 + \frac{z}{p} \right)
  $$
  and by Mertens formula,
  $$
    \sum_{p\leq x} \log\left(1 + \frac{z}{p} \right)
    \leq
    z \sum_{p\leq x}\frac{1}{p}
    =
    z \left(\log \log x +O(1)\right).
  $$
  This shows~\eqref{maj_sum_omega_over_n}.  In order to
  prove~\eqref{maj_sum_omega}, we first write
  $$
    (\log x)\sum_{n\leq x} \mu^2(n)z^{\omega(n)}
    =
     S_1(x)
    +
     S_2(x)
  $$
  where
  $$
    S_1(x) = \sum_{n\leq x} \mu^2(n)z^{\omega(n)}\log n
    \quad
    \mbox{ and }
    \quad
    S_2(x) = \sum_{n\leq x} \mu^2(n)z^{\omega(n)}\log \frac{x}{n}.
  $$
  Since $\mu^2(n)\log n = \mu^2(n)\sum_{p \dv n} \log p$ and
  $\sum_{p\leq X} \log p \ll X$, we obtain
  \begin{align*}
    S_1(x)&= \sum_{p\leq x} \log p \sum_{m\leq x/p} \mu^2(mp)z^{\omega(mp)}
            =
            \sum_{p\leq x} \log p \sum_{\substack{m\leq x/p\\p\nmid m}} \mu^2(m)z^{1+\omega(m)}\\
          &\leq
            \sum_{m\leq x} \mu^2(m)z^{1+\omega(m)} \sum_{p\leq x/m} \log p
            \ll_{z}
            x \sum_{m\leq x} \frac{\mu^2(m)z^{\omega(m)}}{m}.
  \end{align*}
  Moreover since $\log \frac{x}{n} \leq \frac{x}{n}$, the same bound
  also holds for $S_2(x)$.  By~\eqref{maj_sum_omega_over_n}, it
  follows that
  \begin{align*}
    \sum_{n\leq x} \mu^2(n)z^{\omega(n)}
    \ll_z
    \frac{x}{\log x} \sum_{n\leq x} \frac{\mu^2(n)z^{\omega(n)}}{n}
      \ll_z x (\log x)^{z-1}
  \end{align*}
  which establishes~\eqref{maj_sum_omega}.
\end{proof}

\section{Preliminary study of $R_n(d)$}\label{sec:prel-study-r_nd}

We recall that $R_n(d)$ is defined by~\eqref{eq: def Rd} together
with~\eqref{eq:def Td}.  It is convenient to define for
$j \in \{0,1\}$ and $d\geq 1$:
\begin{equation}\label{Rtilde}
  \widetilde{R}_n(d,j)  = \frac{| \Set|}{3^jd^2}
  \sum_{\substack{0< h_1,h_2 <d}} 
  \overline{H(d,h_1,h_2)} 
  \sum_{0\leq \ell <3^j}
  F_n\left(\frac{h_2}{d},-\left(\frac{h_1}{d}+\frac{\ell}{3^j}\right)\right).
\end{equation}
We note that for $j \in \{0,1\}$, $\widetilde{R}_n(1,j)= 0$.

\begin{lemma}\label{Rd}
  For any squarefree integer $d\geq 1$ such that $\gcd(d,6)=1$ and any
  $j\in\{0,1\}$, we have
  \begin{align*}
    R_n(3^jd) =
    \widetilde{R}_n(d,j) + O\left( f(d)\right).
  \end{align*}
\end{lemma}
 
\begin{proof}
  Let $j\in\{0,1\}$.  By multiplicativity, we have
  $$
    \frac{f(3^jd)}{3^jd} = \frac{f(3^j)f(d)}{3^jd} = \frac{f(d)}{3^jd},
  $$
  hence, recalling~\eqref{eq: def Rd}, we obtain
  $$
     R_n(3^jd) = T_n(3^jd) -\frac{f(3^jd)}{3^jd}\,| \Set| 
    = T_n(3^jd)  - \frac{f(d)}{3^jd}\,| \Set|.
  $$
  Since $\mirror{a} \equiv (-1)^{n-1} a \bmod 3$, recalling the
  definition~\eqref{eq:def Td} we may also write
  $$
    T_n(3^jd) = \abs{\{ a\in \Set:~d \dv a \mirror{a} \text{ and } 3^j
    \dv a\}}.
  $$
  We now filter the integers $a\in \Set$ according to the residue
  classes of $a$ and $\mirror{a}$ modulo $d$:
  $$
    T_n(3^jd) = \sum_{\substack{{0\le u,v <d}\\ {uv\equiv 0\bmod d}}} \sum_{\substack{a \in \Set\\a \equiv 0 \bmod 3^j}}
    \mathbf{1}_{a\equiv u\bmod d} \,\mathbf{1}_{\overleftarrow{a}\equiv v\bmod d},
  $$ 
  and using the orthogonality of exponential functions to control the
  conditions $a\equiv u\bmod d$ and
  $ \overleftarrow{a}\equiv v\bmod d$, we obtain
  \begin{align*}
    T_n(3^jd) &=\sum_{\substack{{0\le u,v< d}\\ {d\dv uv}}}
    \sum_{\substack{a\in  \Set\\a\equiv 0 \bmod 3^j}}
    \frac{1}{d^2}\sum_{0\le h_1,h_2 <d}
    \e{ \frac{h_1\(a-u\) + h_2\(\mirror{a}-v\)}{d}}\\ 
              &= T_{n,0}(3^jd) + T_{n,1}(3^jd) + T_{n,2}(3^jd) - T_{n,3}(3^jd)
  \end{align*}
  with
  \begin{align*}
    &  T_{n,0}(3^jd)
      =
      \sum_{\substack{{0\le u,v< d}\\ {d\dv uv}}}
    \sum_{\substack{a\in  \Set\\a\equiv 0 \bmod 3^j}}
    \frac{1}{d^2}\sum_{0< h_1,h_2 <d}
    \e{ \frac{h_1\(a-u\) + h_2\(\mirror{a}-v\)}{d}},\\
    &  T_{n,1}(3^jd)
      =
      \sum_{\substack{{0\le u,v< d}\\ {d\dv uv}}}
    \sum_{\substack{a\in  \Set\\a\equiv 0 \bmod 3^j}}
    \frac{1}{d^2}\sum_{0 \leq h_2 <d}
    \e{ \frac{h_2\(\mirror{a}-v\)}{d}},\\
    &  T_{n,2}(3^jd)
      =
      \sum_{\substack{{0\le u,v< d}\\ {d\dv uv}}}
    \sum_{\substack{a\in  \Set\\a\equiv 0 \bmod 3^j}}
    \frac{1}{d^2}\sum_{0\leq h_1<d}
    \e{ \frac{h_1\(a-u\)}{d}},\\
    &  T_{n,3}(3^jd)
      =
      \sum_{\substack{{0\le u,v< d}\\ {d\dv uv}}}
    \sum_{\substack{a\in  \Set\\a\equiv 0 \bmod 3^j}}
    \frac{1}{d^2}.
 \end{align*}
 By Lemma~\ref{lem_H_f}, we have
 $$
   \sum_{\substack{{0\le u,v< d}\\ {d\dv uv}}} 1 = H(d,0,0) = df(d).
 $$
 It follows that
 $$
   T_{n,3}(3^jd)
   =
   \frac{f(d)}{d}\sum_{\substack{a\in \Set\\a\equiv 0 \bmod 3^j}} 1
   =
   \frac{f(d)}{d} \frac{\abs{\Set}}{3^j} + O\left(\frac{f(d)}{d}\right).
 $$
 Moreover, since $a\equiv 0 \bmod 3^j$ is equivalent to
 $\mirror{a}\equiv 0 \bmod 3^j$, we have
 \begin{align*}
   T_{n,1}(3^jd)
   &=
     \sum_{\substack{{0\le u,v< d}\\ {d\dv uv}}}\frac{1}{d}
   \sum_{\substack{a\in  \Set\\a\equiv 0 \bmod 3^j}}
   \mathbf{1}_{\mirror{a} \equiv v \bmod d}
   = \sum_{\substack{{0\le u,v< d}\\ {d\dv uv}}}\frac{1}{d}
   \left(\frac{\abs{\Set}}{3^jd}+O(1)\right)\\
   &= \frac{f(d)}{d}\frac{\abs{\Set}}{3^j} +O\left(f(d)\right).
 \end{align*}
 Similarly, we obtain
 \begin{align*}
   T_{n,2}(3^jd)
   &= \frac{f(d)}{d}\frac{\abs{\Set}}{3^j} +O\left(f(d)\right).
 \end{align*}
 For $T_{n,0}(3^jd)$, we write
 \begin{align*}
   & T_{n,0}(3^jd)\\
   &\qquad = \frac{1}{d^2}\sum_{0< h_1,h_2 <d}\sum_{\substack{{0\le u,v< d}\\ {d\dv uv}}}
   \e{ \frac{-h_1u-h_2v}{d}}\sum_{\substack{a\in  \Set\\a\equiv 0 \bmod 3^j}}
   \e{\frac{h_1a+h_2\mirror{a}}{d}}\\
   &\qquad =\frac{1}{d^2}\sum_{0< h_1,h_2 <d}\overline{H(d,h_1,h_2)}
     \sum_{\substack{a\in  \Set\\a\equiv 0 \bmod 3^j}}
   \e{\frac{h_1a+h_2\mirror{a}}{d}}
 \end{align*}
 where for any $h_1,h_2$,
 \begin{align*}
   \sum_{\substack{a\in  \Set\\a\equiv 0 \bmod 3^j}}
   \e{\frac{h_1a+h_2\mirror{a}}{d}}
   &= \frac{1}{3^j} \sum_{0\leq \ell <3^j}\sum_{a\in  \Set}
     \e{\frac{h_1a+h_2\mirror{a}}{d}+\frac{\ell a}{3^j}}\\
   &= \frac{| \Set|}{3^j} \sum_{0\leq \ell <3^j}
     F_n\left(\frac{h_2}{d},-\left(\frac{h_1}{d}
     +\frac{\ell}{3^j}\right)\right)
 \end{align*}
 so that
 $$
   T_{n,0}(3^jd) = \widetilde{R}_n(d,j).
 $$
 This completes the proof.
\end{proof}

\begin{lemma}\label{to_show_R_tilde}
  For any real number $D \geq 2$, we have
  \begin{align*}
    \sum_{\substack{d < D\\d \text{ odd}}}  \mu^2(d)4^{\omega(d)}& |R_n(d)|\\
                                                                 &   \ll
                                                                   D(\log D)^7 +
                                                                   \max_{j\in \{0,1\}}
                                                                   \sum_{\substack{d < D\\\gcd(d,6)=1}} \mu^2(d)4^{\omega(d)}|\widetilde{R}_n(d,j)|.
  \end{align*}
\end{lemma}

\begin{proof}
  By splitting the sum over $d$ according to $\gcd(d,3)$, we obtain
  \begin{align*}
    \sum_{\substack{d < D\\\mu^2(d)=1\\d \text{ odd}}} 4^{\omega(d)}|R_n(d)|
    &= \sum_{\substack{d < D\\\mu^2(d)=1\\\gcd(d,6)=1}} 4^{\omega(d)}|R_n(d)|
    + \sum_{\substack{d < D/3\\\mu^2(d)=1\\\gcd(d,6)=1}} 4^{\omega(3d)}|R_n(3d)|\\
    &\leq
      5\max_{j\in\{0,1\}} 
      \sum_{\substack{d < D\\\mu^2(d)=1\\\gcd(d,6)=1}} 4^{\omega(d)}|R_n(3^jd)|.
  \end{align*}
  It follows from Lemma~\ref{Rd} that for $j\in\{0,1\}$,
  $$
    \sum_{\substack{d < D\\\mu^2(d)=1\\\gcd(d,6)=1}} 4^{\omega(d)}|R_n(3^jd)|
    \ll
    \sum_{\substack{d < D\\\mu^2(d)=1\\\gcd(d,6)=1}} 4^{\omega(d)}|\widetilde{R}_n(d,j)|
    + \sum_{\substack{d < D\\\mu^2(d)=1}} 4^{\omega(d)}f(d).
  $$
  Since $f(d)< 2^{\omega(d)}$ for any squarefree integer $d$ (this
  follows from the multiplicativity of $f$ and $f(p) < 2$), we have
  by~\eqref{maj_sum_omega},
  $$
    \sum_{\substack{d < D\\\mu^2(d)=1}} 4^{\omega(d)}f(d)
    \leq
    \sum_{\substack{d < D\\\mu^2(d)=1}} 8^{\omega(d)}
    \ll
    D (\log D)^7.
  $$
  Combining the previous estimates, we get
  $$
    \sum_{\substack{d < D\\\mu^2(d)=1\\d \text{ odd}}} 4^{\omega(d)}|R_n(d)|
    \ll
    D (\log D)^7 +
    \max_{j\in\{0,1\}}
    \sum_{\substack{d < D\\\mu^2(d)=1\\\gcd(d,6)=1}} 4^{\omega(d)}|\widetilde{R}_n(d,j)|,
  $$
  as desired.
\end{proof}

\section{Proof of  Lemma~\ref{lemma:sieve-error-term}}
\label{sec:proof_lemma_1}

Let $0 < \xi < 1/2$ and $D = 2^{\xi n}$.  Let $j\in \{0,1\}$
and $0 \leq \ell < 3^j$. By Lemma~\ref{to_show_R_tilde} and
\eqref{Rtilde}, it suffices to show that there exists $c>0$, which
depends only on $\xi$, such that
$$
  E(n,D,j,\ell)
  \ll_{\xi}
  \exp(-c \sqrt{n})
$$
where
\begin{align*}
  &E(n,D,j,\ell)\\
  & \qquad =
    \sum_{\substack{d \leq D\\\mu^2(d)=1\\\gcd(d,6)=1}}
  \frac{4^{\omega(d)}}{d^2}
  \sum_{0< h_1,h_2 <d}
  \abs{H(d,h_1,h_2)}
  \abs{F_n\left(\frac{h_2}{d},-\left(\frac{h_1}{d}+\frac{\ell}{3^j}\right)\right)}.
\end{align*}
In order to handle the common factors of $d$, $h_1$ and $h_2$, we
write $d=d_0d_1d_2d_3$ with suitable $d_i$. More precisely, we show
that the set of summation $\mathcal{E}(D)$ of all $(d,h_1,h_2)\in\N^3$
satisfying
\begin{gather*}
  1\leq d \leq D, \ \mu^2(d)=1,\ \gcd(d,6)=1,\\
  0< h_1 <d,\; 0< h_2 <d
\end{gather*}  
may be replaced by the set $\mathcal{F}(D)$ of all
$(d_0,d_1,d_2,d_3,h_1^*,h_2^*)\in\N^6$ satisfying
\begin{gather*}
  1\leq d_0d_1d_2d_3 \leq D, \ \mu^2(d_0d_1d_2d_3)=1,\ \gcd(d_0d_1d_2d_3,6)=1,\\
  0< h_1^*<d_2d_3, \ 0< h_2^*<d_1d_3, \ \gcd(h_1^*,d_2d_3) = 1, \
  \gcd(h_2^*,d_1d_3) = 1.
\end{gather*}
We easily check that the map
$$
  \begin{array}{clc}
    \mathcal{E}(D) &\to & \mathcal{F}(D)\\
    (d,h_1,h_2)& \mapsto & (d_0,d_1,d_2,d_3,h_1^*,h_2^*)
  \end{array}
$$
where $d_0,d_1,d_2,d_3,h_1^*,h_2^*$ are defined by
\begin{gather*}
  d_0 = \gcd(d,h_1,h_2), \  d_0d_1 = \gcd(d,h_1), \  d_0d_2 = \gcd(d,h_2),\\
  d_0d_1d_2d_3 = d,\ h_1^*d_0d_1 = h_1,\ h_2^*d_0d_2 = h_2,
\end{gather*}
is well defined and bijective with inverse
$$
  \begin{array}{clc}
    \mathcal{F}(D) &\to & \mathcal{E}(D)\\
    (d_0,d_1,d_2,d_3,h_1^*,h_2^*)& \mapsto & (d_0d_1d_2d_3, h_1^*d_0d_1, h_2^*d_0d_2).
  \end{array}
$$
It follows that
\begin{align*}
  E(n,D,j,\ell)
  & =
    \sum_{(d_0,d_1,d_2,d_3,h_1^*,h_2^*)\in \mathcal{F}(D)}
    \frac{4^{\omega(d_0d_1d_2d_3)}}{(d_0d_1d_2d_3)^2}\\
  &   \qquad \qquad  \qquad\times  \abs{H(d_0d_1d_2d_3,h_1^*d_0d_1,h_2^*d_0d_2)}\\
  & \qquad  \qquad  \qquad \qquad\times \abs{F_n\left(\frac{h_2^*}{d_1d_3},-\left(\frac{h_1^*}{d_2d_3}+\frac{\ell}{3^j}\right)\right)}.
\end{align*}
Moreover, Lemma~\ref{H} implies that for
$(d_0,d_1,d_2,d_3,h_1^*,h_2^*)\in \mathcal{F}(D)$,
$$
  \abs{H\left (d_0d_1d_2d_3,h_1^*d_0d_1,h_2^*d_0d_2\right )} \leq 2^{\omega(d_0)}d_0d_1d_2.
$$
Hence
$$
  E(n,D,j,\ell)
  \leq
  E_1(n,D,j,\ell), 
$$
where
\begin{align*}
  E_1(&n,D,j,\ell) \\
  &=
    \sum_{(d_0,d_1,d_2,d_3,h_1^*,h_2^*)\in \mathcal{F}(D)} 
    \frac{8^{\omega(d_0)} 4^{\omega(d_1d_2d_3)}}{d_0d_1d_2d_3^2}
    \abs{F_n\left(\frac{h_2^*}{d_1d_3},-\left(\frac{h_1^*}{d_2d_3}+\frac{\ell}{3^j}\right)\right)}.
\end{align*}
We write $E_1(n,D,j,\ell)=E_2(n,D,j,\ell)+E_3(n,D,j,\ell)$, where in
the sum $E_2(n,D,j,\ell)$ we have $d_1d_2d_3 \leq W$ and in the sum
$E_3(n,D,j,\ell)$, we have $d_1d_2d_3 > W$, where $W$ is a parameter
to be precised such that
$$
  2 \leq W \leq D.
$$
For $(d_0,d_1,d_2,d_3,h_1^*,h_2^*)\in \mathcal{F}(D)$ with
$d_1d_2d_3 \leq W$, by Lemma~\ref{lem_maj_Fn}, we have
$$
  \abs{F_n\left(\frac{h_2^*}{d_1d_3},-\left(\frac{h_1^*}{d_2d_3}+\frac{\ell}{3^j}\right)\right)}
  \ll
  \exp\left(\frac{-c_0\,n}{\log\left(\frac{4W}{3}\right)}\right).
$$
It follows that 
$$
  E_2(n,D,j,\ell)
  \ll
  \exp\left(\frac{-c_0\,n}{\log\left(\frac{4W}{3}\right)}\right)
  \sum_{\substack{1\leq d_0d_1d_2d_3 \leq D\\d_1d_2d_3\leq W\\\mu^2 (d_0d_1d_2d_3)=1}}
  \frac{8^{\omega(d_0)} 4^{\omega(d_1d_2d_3)}}{d_0}, 
$$
where, by Lemma~\ref{lem_maj_average_omega}, the sum in the right-hand
side is at most
\begin{align*}
  \sum_{\substack{d_1\leq W\\\mu^2(d_1)=1}} &4^{\omega(d_1)} 
                                              \sum_{\substack{d_2\leq W/d_1\\\mu^2(d_2)=1}} 4^{\omega(d_2)}
  \sum_{\substack{d_3\leq W/(d_1d_2)\\\mu^2(d_3)=1}} 4^{\omega(d_3)}
  \sum_{\substack{d_0\leq D\\\mu^2(d_0)=1}} \frac{8^{\omega(d_0)}}{d_0}\\
                                            &\qquad \ll
                                              W (\log W)^3 (\log D)^8
                                              \sum_{\substack{d_1\leq W\\\mu^2(d_1)=1}} \frac{4^{\omega(d_1)}}{d_1}
  \sum_{\substack{d_2\leq W\\\mu^2(d_2)=1}} \frac{4^{\omega(d_2)}}{d_2}\\
                                            &\qquad \ll
                                              W (\log W)^{11} (\log D)^8
                                              \ll W (\log D)^{19}.
\end{align*}
Therefore,
$$
  E_2(n,D,j,\ell)
  \ll
  W
  (\log D)^{19}
  \exp\left(\frac{-c_0\,n}{\log\left(\frac{4W}{3}\right)}\right)
  .
$$
In order to bound $E_3(n,D,j,\ell)$, we first proceed to some dyading
splitting in $d_1,d_2,d_3$ so that $d_i\in [D_i, 2D_i)$ for $i=1,2,3$
with
$$
  W/8 < D_1D_2D_3\leq D
$$ 
and we relax the conditions $\mu^2 (d_0d_1d_2d_3)=1$ and
$\gcd(d_0d_1d_2d_3,6)=1$ to keep only $\mu^2(d_0)=1$ and
$\gcd(d_1d_2d_3,6)=1$:
\begin{align*}
  &   E_3(n,D,j,\ell)\\
  & \qquad   \ll 
    (\log D)^3 \sum_{\substack{d_0\leq D\\\mu^2(d_0)=1}}
  \frac{8^{\omega(d_0)}}{d_0}
  \max_{\substack{(D_1,D_2,D_3)\in\N^3\\W/8 < D_1D_2D_3 \leq D}}
  \sum_{\substack{d_1,d_2,d_3\\d_i \sim D_i,\,
  i=1,2,3\\\gcd(d_1d_2d_3,6)=1}}
  \frac{4^{\omega(d_1d_2d_3)}}{d_1d_2d_3^2}\\
  & \qquad \qquad   \quad \times
    \sum_{\substack{0 < h_1^* < d_2d_3\\\gcd(h_1^*,d_2d_3)=1}}
  \sum_{\substack{0 < h_2^* < d_1d_3\\\gcd(h_2^*,d_1d_3)=1}}
  \abs{F_n\left(\frac{h_2^*}{d_1d_3},-\left(\frac{h_1^*}{d_2d_3}+\frac{\ell}{3^j}\right)\right)}.
\end{align*}
It follows from~\eqref{maj_sum_omega_over_n} and~\eqref{eq:tau tautilde}
that
\begin{align*}
  & E_3(n,D,j,\ell) \\
  & \qquad  \quad   \ll
    (\log D)^{11} \max_{\substack{\cD = (D_1,D_2,D_3)\in\N^3\\
  W/8 < D_1D_2D_3 \leq D}} \tautilde(8D_1D_2D_3)^2\,
  \frac{M_1(n;\cD,(0,3^{1-j}\ell))}{D_1D_2D_3^2}, 
\end{align*}
where $M_1(n;\cD,(0,3^{1-j}\ell))$ is defined
by~\eqref{notation_R_D_l}.  If $D_1D_2D_3 \leq D$ then, since
$D = 2^{\xi n}$, the condition~\eqref{eq:condD1D2D3} is satisfied
for $n\geq n_0(\xi)$ with the choice 
$$
  \varepsilon = \varepsilon(\xi)=
  \frac{1}{2}\min\left\{2,\frac{1}{2\xi} - 1\right\}\in\left(0,1\right].
$$
Thus by Lemma~\ref{lem:maj_R_D_l},
$$
  \frac{M_1(n;\cD,(0,3^{1-j}\ell))}{D_1D_2D_3^2}
   \ll \tautilde(4D_1D_3)^{1/2}\, \tautilde(4D_2D_3)^{1/2}
  \left(D_1D_2D_3\right)^{-\varepsilon \eta_0}.
$$
Since $\tautilde(d) \ll_{\xi} d^{\varepsilon \eta_0/6}$ for any positive
integer $d$ (see for instance~\cite[Theorem~315]{hardy-wright-1979})
and recalling that $D_1D_2D_3>W/8$, we obtain
$$
  E_3(n,D,j,\ell)
  \ll_{\xi}
  (\log D)^{11} W^{-\varepsilon \eta_0/2}
  .
$$
Combining the previous estimates, we get
$$
  E(n,D,j,\ell)
  \ll_{\xi}
  W
  (\log D)^{19}
  \exp\left(\frac{-c_0\,n}{\log\left(\frac{4W}{3}\right)}\right)
  +
  (\log D)^{11} W^{-\varepsilon \eta_0/2}
  .
$$
By choosing
$$
  W = \exp\left(\delta n^{1/2}\right)
  \quad
  \text{ with }
  \quad
  \delta = \delta(\xi)=
  c_0^{1/2}\left(\frac{\varepsilon(\xi)\eta_0}{2}+1\right)^{-1/2} >0
  ,
$$
we obtain
\begin{align*}
  E(n,D,j,\ell)
  &\ll_{\xi}
    n^{19}
    \exp\left(-\left(\frac{c_0}{\delta} - \delta \right)n^{1/2}\right)
    +
    n^{11} \exp\left( -\frac{\delta\varepsilon \eta_0}{2}  n^{1/2}\right)\\
  &\qquad =
    \left(n^{19} + n^{11}\right)
    \exp\left(-\widetilde{\delta} \, n^{1/2}\right)
\end{align*}
where $\widetilde{\delta} = \delta\varepsilon \eta_0 /2 >0$ depends only on $\xi$.
It follows that
$$
  E(n,D,j,\ell) \ll_{\xi} \exp\left(-\frac{\widetilde{\delta}}{2}  n^{1/2}\right),
$$
which completes the proof of Lemma~\ref{lemma:sieve-error-term}.

\section{Proofs of main results}\label{completion_proofs}

\subsection{Proof of Theorem~\ref{thm:Omega-z}}

We recall the notations introduced in Section~\ref{sec_approach}.

Let $0<\gamma < 1/(2\beta_2)$. We define
$\xi = \frac{1}{2}(\beta_2\gamma + \frac{1}{2})$ so that
$\beta_2\gamma < \xi < 1/2$ and let $z = 2^{\gamma n}$ and
$y = 2^{\xi n}\geq z$.

By Lemma~\ref{lemma:sieve-error-term}, there exists $c>0$, which
depends only on $\gamma$, such that
$$
  \sum_{\substack{d \dv P(z)\\d<y}}
  4^{\omega(d)}|R_n(d)|
  \leq
  \sum_{\substack{d < y\\d \text{ odd}}}
  \mu^2(d)  4^{\omega(d)}|R_n(d)|
  \ll_{\gamma} 2^n \exp(-c\sqrt{n}).
$$
Since
$h^+\left(\frac{\log y}{\log
    z}\right)=h^+\left(\frac{\xi}{\gamma}\right) > 0$, it follows
from~\eqref{eq:Sieve-UB} that
$$
  \varTheta\left(n,z\right)
  \ll_{\gamma} \frac{2^{n}}{n^2}.
$$
Moreover, since $\beta_2\gamma < \xi$, we have
$h^-\left(\frac{\log y}{\log
    z}\right)=h^-\left(\frac{\xi}{\gamma}\right) > 0$ and it
follows from~\eqref{eq:Sieve-LB} that there exists $n_0 = n_0(\gamma)$
such that for $n\geq n_0$, we have
$$
  \varTheta\left(n,z\right)
  \gg_{\gamma} \frac{2^{n}}{n^2}.
$$
This completes the proof of Theorem~\ref{thm:Omega-z}.

\subsection{Proof of Theorem~\ref{thm-squarefree-v2}}
We first detect the condition that $\mirror{a}$ is squarefree:
$$
  \SQF(n)=\sum_{a\in \Set }\mu^2 (a)\sum_{d_1^2|\mirror{a}}\mu (d_1).
$$
Let $D_1\le 2^{n/2}$ be a parameter to be precised. We split the sum
in $\SQF(n) = S_{11}+S_{12}$, where $d_1\le D_1$ in $S_{11}$ and
$d_1>D_1$ in $S_{12}$.  For the sum $S_{12}$, we forget the condition
``$a$ is squarefree'', and reverse the roles between $a$ and
$\mirror{a}$:
\begin{align*}
  \abs{S_{12}} & = \abs{\sum_{D_1<d_1< 2^{n/2}}\mu (d_1)
                 \sum_{\substack{{a\in \Set }\\ {d_1^2 |\mirror{a}}}}\mu^2 (a)}\\
               &  \le \sum_{D_1<d_1<2^{n/2}}\sum_{\substack{{a\in \Set }\\ {d_1^2|a}}}1
  \ll \sum_{d_1>D_1}\frac{|\Set |}{d_1^2}\ll \frac{|\Set |}{D_1}.
\end{align*}
We now detect in $S_{11}$ the condition $\mu^2 (a)=1$:
$$
  S_{11}
  =\sum_{d_1\le D_1}\mu (d_1)
  \sum_{\substack{{a\in \Set }\\ {d_1^2 |\mirror{a}}}}\mu^2 (a)
  =\sum_{d_1\le D_1}\mu (d_1)\sum_{\substack{{a\in \Set }\\ {d_1^2
        |\mirror{a}}}}\sum_{d_2^2|a}\mu (d_2).
$$
We introduce a parameter $D_2$ such that $D_1 \leq D_2\le 2^{n/2}$ and
split $S_{11}$ in $S_{11} = S_{21}+S_{22}$ where $d_2\le D_2$ in
$S_{21}$ and $d_2>D_2$ in $S_{22}$.  We bound trivially $S_{22}$:
$$
  \abs{S_{22}}
  \le D_1\sum_{d_2 >D_2}\sum_{\substack{{a\in \Set }\\ {d_2^2 |a}}}1
  \ll |\Set |\, \frac{D_1}{D_2}.
$$
It remains to estimate
$$
  S_{21}
  =\sum_{\substack{{d_1\le D_1}\\ {d_2\le D_2}\\ {\gcd(d_1d_2,2)=1}}}\mu(d_1)\mu(d_2)
  \sum_{\substack{a\in \Set \\ {d_1^2 \dv \mirror{a}}\\ {d_2^2 \dv a}}}1.
$$
Since the inner sum above is for $a\in \Set $, it is licit to add the
condition $\gcd(d_1d_2,2)=1$.  We detect the conditions
$d_1^2 \dv \mirror{a}$ and $d_2^2 \dv a$ via exponential sums.  The
sum $S_{21}$ becomes:
\begin{align*}
  S_{21} &=\sum_{\substack{{d_1\le D_1}\\ {d_2\le D_2}\\
  {\gcd(d_1d_2,2)=1}}} \frac{\mu(d_1)\mu(d_2)}{d_1^2d_2^2}\sum_{a\in
  \Set } \sum_{0\leq h_1<d_1^2}\sum_{0\leq h_2<d_2^2}
  \e{\frac{h_1\mirror{a}}{d_1^2}+\frac{h_2a}{d_2^2}}\\
         &=|\Set | \sum_{\substack{{d_1\le D_1}\\ {d_2\le D_2}\\
  {\gcd(d_1d_2,2)=1}}}
  \frac{\mu(d_1)\mu(d_2)}{d_1^2d_2^2}\sum_{0\leq
  h_1<d_1^2}\sum_{0\leq h_2<d_2^2}
  F_n\left(\frac{h_1}{d_1^2},\frac{-h_2}{d_2^2}\right).
\end{align*}
We now split this sum in $S_{21} = S_{31} + S_{32}$ where, in
$S_{31}$, $d_1^2 \notdv 3h_1$ or $d_2^2 \notdv 3h_2$ and in $S_{32}$,
$d_1^2 \dv 3h_1$ and $d_2^2 \dv 3h_2$.  By Lemma~\ref{lem_maj_Fn}, we
have
\begin{align*}
  \abs{S_{31}}
  &\leq  |\Set | D_1D_2
    \max_{\substack{(d_1,d_2,h_1,h_2)\\
  d_i \leq D_i, 0\leq h_i < d_i^2, i=1,2\\ \gcd(d_1d_2,2)=1
  \\d_1^2 \notdv 3h_1 \text{ or } d_2^2 \notdv 3h_2}}
  \abs{F_n\left(\frac{h_1}{d_1^2},\frac{-h_2}{d_2^2}\right)}\\
  &\ll |\Set | D_1D_2 \exp\left(\frac{-c_0n}{\log(\frac{4D_2^2}{3})}\right).
\end{align*}
It remains to estimate $S_{32}$. We split the sum according to the
value of
$\left(\frac{3h_1}{d_1^2},\frac{3h_2}{d_2^2}\right) \in \{0,1,2\}^2$:
$$
  S_{32}
  = |\Set | \sum_{0\leq u_1,u_2 \leq 2} F_n\left(\frac{u_1}{3},\frac{-u_2}{3}\right)
    \prod_{i=1}^2\left(\sum_{\substack{d_i \leq D_i\\\gcd(d_i,2)=1,\ 3 \dv u_id_i^2}}
  \frac{\mu(d_i)}{d_i^2}\right).
$$
For $i\in \{1,2\}$ and $u_i \in \{0,1,2\}$, we easily check that
$$
  \sum_{\substack{d_i \leq D_i\\\gcd(d_i,2)=1,\ 3 \dv u_id_i^2}}
  \frac{\mu(d_i)}{d_i^2} = \frac{c(u_i)}{\zeta(2)} +O\left(\frac{1}{D_i}\right), 
$$
where $c(u_i)$ is defined by $c(0) = 4/3$ and $c(1) = c(2) = -1/6$.
It follows that
\begin{align*}
  S_{32}
  &= |\Set | \sum_{0\leq u_1,u_2 \leq 2}
    F_n\left(\frac{u_1}{3},\frac{-u_2}{3}\right)
    \left(\frac{c(u_1)c(u_2)}{\zeta(2)^2}
    + O\left(\frac{1}{D_1}\right)\right)
  \\
  &= \frac{S_{41}}{\zeta(2)^2}
    + O\left(\frac{|\Set |}{D_1}\right)
\end{align*}
where $S_{41}$ is defined by
$$
  S_{41} = |\Set |\sum_{0\leq u_1,u_2 \leq 2}
    F_n\left(\frac{u_1}{3},\frac{-u_2}{3}\right)c(u_1)c(u_2).
$$
To estimate the sum over $(u_1,u_2)$, we use again exponential sums:
\begin{align*}
  S_{41}
  &= \sum_{0\leq u_1,u_2 \leq 2}c(u_1)c(u_2)\sum_{a \in \Set }
    \e{\frac{u_1\mirror{a}}{3}+\frac{u_2a}{3}}\\
  &= \sum_{a\in \Set }
    \left(\sum_{0\leq u_1 \leq 2} c(u_1)\e{\frac{u_1\mirror{a}}{3}}\right)
    \left(\sum_{0\leq u_2 \leq 2} c(u_2)\e{\frac{u_2a}{3}}\right).
\end{align*}
For the inner sum over $u_1$, we write
\begin{align*}
  \sum_{0\leq u_1 \leq 2} c(u_1)\e{\frac{u_1\mirror{a}}{3}}
  &= c(0) + c(1)\left(\e{\frac{\mirror{a}}{3}}
    + \e{\frac{2\mirror{a}}{3}}\right)\\
  &=  c(0) + c(1) \left( 3 \cdot\mathbf{1}_{3 \dv \mirror{a}}-1\right)\\
  &= \frac{1}{2} \left(3 -\mathbf{1}_{3 \dv \mirror{a}}\right)
\end{align*}
and similarly for the sum over $u_2$. Since
$$
  \mirror{a} \equiv (-1)^{n-1} a \bmod 3,
$$
it follows that
$$
  S_{41}
  = \frac{1}{4}\sum_{a\in \Set } \left(3 -\mathbf{1}_{3 \dv a}\right)^2
  = \frac{1}{4}\left(9|\Set | - 5 \sum_{\substack{a \in \Set \\3 \dv a}} 1\right)
  = \frac{11}{6} \, |\Set | + O(1).
$$
Combining the previous estimates, we finally obtain
$$
  \SQF(n)
  = |\Set | \left(
    \frac{11}{6\zeta(2)^2}
    + O\left(
      \frac{D_1}{D_2}
      + \frac{1}{D_1}
      + D_1D_2
      \exp\left(\frac{-c_0n}{\log(\frac{4D_2^2}{3})}\right)
      \right)
  \right).
$$
By choosing $D_2 = D_1^2 = \exp\left( \frac{1}{2} c_0^{1/2}  n^{1/2}\right)$,
we get
$$
  \SQF(n)
  = |\Set | \left(
    \frac{11}{6\zeta(2)^2} +O\left(\exp\left(-c n^{1/2}\right)\right)
  \right)
$$
with $c = \frac{1}{4} c_0^{1/2}= 0.0439...$, which
proves~\eqref{asymp_squarefree_Bn}.

If $b$ is an $(n-1)$-bit integer then $2b$ is an $n$-bit integer and
$\mirror{2b}= \mirror{b}$.  It follows that the number of even
integers $a$ such that $2^{n-1}\leq a < 2^n$ and
$\mu ^2(a)=\mu^2(\mirror{a})=1$ is
$$
  \abs{\{2^{n-2}\leq b < 2^{n-1} :~\mu ^2(2b)=\mu^2(\mirror{b})=1\}} = \SQF(n-1).
$$
Since $|\Set| = 2^{n-2}$, this implies that
$$
    \widetilde \SQF(n) = \SQF(n) + \SQF(n-1)
    =2^{n-1}\left(\frac{11}{8\zeta(2)^2} + O(\exp(-c n^{1/2}))\right), 
$$
which proves~\eqref{asymp_squarefree}.

\section{Numerical investigations on the number of reversible
  primes}\label{sec_count_emirps}

\subsection{Preambule} 
In this section, we provide numerical investigations on reversible
primes. In fact, we do not restrict ourselves to base $2$ and provide
also some numerical data in base $10$.

\subsection{Base 2}

We recall that~\cite[\href{http://oeis.org/A074831}{A074831}]{OEIS}
provides a table of the number of binary reversible primes less than
$10^m$ for $m \leq 12$. We think more useful to provide a table of the
number $\varTheta(n)$ of binary reversible primes with exactly $n$
binary digits and we do so for $n \leq 50$ in Table~\ref{table_2}.

We have written a program that combines a classical Eratosthenes sieve
(optimized using the library \texttt{primesieve}) with a variant of
Eratosthenes sieve in residue classes.  This permits to organize the
computations by blocks so that the tables fit into the computer memory
and produce a considerable speedup due to a strong use of the
\texttt{L1} memory cache.

\begin{table}[h!]
  \makebox[\linewidth]{
{\tiny \begin{tabular}{|c|c||c|c||c|c||c|c||c|c|}
      \hline
      $n$ & $\varTheta(n)$ & $n$ & $\varTheta(n)$ & $n$ & $\varTheta(n)$ & $n$ & $\varTheta(n)$ & $n$ & $\varTheta(n)$ \\
      \hline
      1   & 0                & 11  & 69               & 21  & 16732            & 31  & 7377931          & 41  & 4222570054       \\
      \hline
      2   & 1                & 12  & 94               & 22  & 29392            & 32  & 13878622         & 42  & 8056984176       \\ 
      \hline
      3   & 2                & 13  & 178              & 23  & 55109            & 33  & 25958590         & 43  & 15315267089      \\ 
      \hline
      4   & 2                & 14  & 308              & 24  & 101120           & 34  & 48421044         & 44  & 29274821854      \\ 
      \hline
      5   & 4                & 15  & 589              & 25  & 179654           & 35  & 92163237         & 45  & 55976669028      \\ 
      \hline
      6   & 6                & 16  & 908              & 26  & 332130           & 36  & 173672988        & 46  & 106505783902     \\ 
      \hline
      7   & 9                & 17  & 1540             & 27  & 625928           & 37  & 325098134        & 47  & 204628057694     \\ 
      \hline
      8   & 14               & 18  & 2814             & 28  & 1136814          & 38  & 617741968        & 48  & 392422557460     \\ 
      \hline
      9   & 27               & 19  & 5158             & 29  & 2120399          & 39  & 1177573074       & 49  & 749026893680     \\ 
      \hline
      10  & 36               & 20  & 9210             & 30  & 3963166          & 40  & 2221353224       & 50  & 1440435348050    \\
      \hline
    \end{tabular}
  } 
  }
  \caption{Base 2: values of $\varTheta(n)$ for $n\leq 50$.}
  \label{table_2}
\end{table}

A table of the number of binary prime palindromes with $n$ binary
digits is given in~\cite[\href{http://oeis.org/A117773}{A117773}]{OEIS}.
Our calculations give us the opportunity to confirm these values for
all $n\leq 50$.

Let us describe a heuristic argument that leads us to a conjecture on
the asymptotic behaviour of $\varTheta(n)$.

Let $a$ be a randomly chosen integer in $\Set$. If $a$ is prime then
$3 \nmid a$, which is equivalent to $3 \nmid \mirror{a}$
by~\eqref{congr_mod_3}.  Therefore the events ``$a$ is prime'' and
``$\mirror{a}$ is prime'' are not expected to be ``independent" but it
is natural to expect that they are ``conditionally independent'' given
that $3 \nmid a$. This would imply that
\begin{align*}
  &   \mathbb{P}\left(a \text{ and }\mirror{a} \text{ are prime}\right)\\
  &\qquad\qquad\qquad  = \mathbb{P}\left(a \text{ and }\mirror{a} \text{ are prime}\,|\, 3 \nmid a\right)\mathbb{P}\left(3 \nmid a\right)\\
  &\qquad\qquad\qquad\approx
    \mathbb{P}\left(a \text{ is prime}\,|\, 3 \nmid a\right)
    \mathbb{P}\left(\mirror{a} \text{ is prime}\,|\, 3 \nmid a\right)
    \mathbb{P}\left(3 \nmid a\right)\\
  &\qquad\qquad\qquad\quad=
    \frac{\mathbb{P}(a \text{ is prime})\,
    \mathbb{P}(\mirror{a} \text{ is prime})}
    {\mathbb{P}\left(3 \nmid a\right)}.
\end{align*}
Moreover
$\mathbb{P}(a \text{ is prime}) = \mathbb{P}(\mirror{a} \text{ is
  prime})$
and since
$a\in\Set$, by the Prime Number Theorem, we have
$$
  \mathbb{P}(a \text{ is prime}) 
   =\frac{\Li(2^n)-\Li(2^{n-1})}{|\Set|} (1+o(1))\qquad (n \to \infty), 
$$
where 
$$
\Li(x) = \int_2^x \frac{d t}{\log t}.
$$
Thus we may expect that
$$
  \varTheta(n)
  = |\Set|\,
  \mathbb{P}\left(a \text{ and }\mirror{a} \text{ are prime}\right)
  \approx
  \varTheta_{\text{exp}}(n)
$$
where
$$
\varTheta_{\text{exp}}(n)
=  \frac{3\left(\Li(2^n)-\Li(2^{n-1})\right)^2}{2^{n-1}}
= (3+o(1))\frac{2^{n-1}}{(\log 2^n)^2}\qquad (n \to \infty)
.
$$

This agrees with the values of $\varTheta(n)$ provided in
Table~\ref{table_2}, as illustrated graphically by Figure~\ref{plot}.

\begin{figure}[h!]
  \centering \includegraphics[width=9cm]{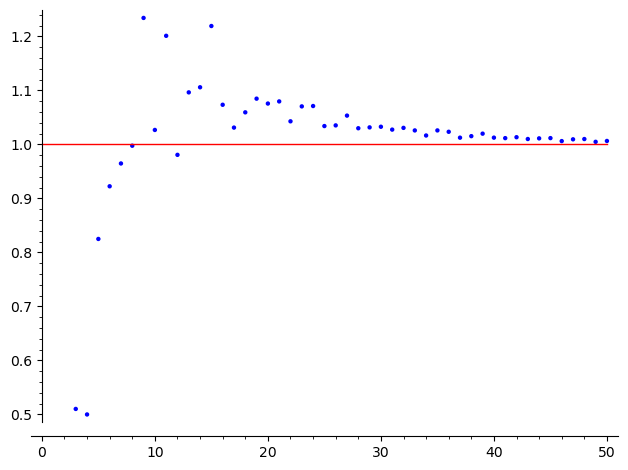}
  \caption{Base 2: graph of $\varTheta(n)/\varTheta_{\text{exp}}(n)$ for $n\leq 50$.}
  \label{plot}
\end{figure}  
This leads us to formulate the following. 

\begin{conjecture}
  $$
  \varTheta(n) =(3+o(1)) \frac{2^{n-1}}{\left(\log 2^n\right)^2}
  \qquad (n \to \infty).
  $$
\end{conjecture}

\subsection{Base 10}

In this section, we consider $n$-digits integers $k$ in base 10 such
that
$$
  k = \sum_{j=0}^{n-1} \varepsilon_{j}(k) \,10^j,
  \quad
  \mirror{k} = \sum_{j=0}^{n-1} \eps_j(k)\,10^{n-1-j}
$$
where $\varepsilon_j(k) \in \{0,1,\ldots,10\}$,
$j\in\{0, \ldots, n-1\}$, $ \varepsilon_{n-1}(k)\neq 0$.  We define
$$
  \varTheta_{10}(n)
  =
  \abs{\left\{10^{n-1} \leq p < 10^n : ~p \text{ and } \mirror{p} \text{ are prime} \right\}}.
$$

We note that~\cite[\href{http://oeis.org/A048054}{A048054}]{OEIS}
provides a table of the number $\varTheta_{10}(n)$ of base 10
reversible primes with exactly $n$ digits in base 10 for $n \leq
13$. We extend this to $n\leq 15$ in Table~\ref{table_10}.

\begin{table}[h!]
  \begin{center}
  {\tiny   \begin{tabular}{|c|c||c|c||c|c|}
      \hline
      $n$ & $\varTheta_{10}(n)$ & $n$ & $\varTheta_{10}(n)$ & $n$ & $\varTheta_{10}(n)$ \\
      \hline 
      1   & 4                   & 6   & 9538                & 11  & 274932272           \\
      \hline
      2   & 9                   & 7   & 71142               & 12  & 2294771254          \\
      \hline
      3   & 43                  & 8   & 535578              & 13  & 19489886063         \\
      \hline
      4   & 204                 & 9   & 4197196             & 14  & 167630912672        \\
      \hline
      5   & 1499                & 10  & 33619380            & 15  & 1456476399463       \\
      \hline
    \end{tabular}
    }
  \end{center}
  \caption{Base 10: values of $\varTheta_{10}(n)$ for $n\leq 15$.}
  \label{table_10}
\end{table}

Furthermore,~\cite[\href{http://oeis.org/A016115}{A016115}]{OEIS}  
provides a table of the number of base 10 prime palindromes with $n$
digits.  Our calculations give us the opportunity to confirm these
values for all $n\leq 15$.

\section*{Acknowledgement}

This work was motivated by conversations with Pieter Moree, who we
would like to thank.  The authors also express their gratitude to
their late colleague and friend Christian Mauduit with who several possible
approaches to this project were discussed in 2014 in Luminy.

During the preparation of this work C.D., B.M. and J.R. were supported
by ANR-FWF Grant 4945-N and ANR Grant 20-CE91-0006 and I.S.  by the
ARC Grants DP230100530 and DP230100534.


\begin{thebibliography}{99}

\bibitem{Banks-Hart-Sakata}{\sc W. D.~Banks, D.~Hart and M.~Sakata},
  {\em Almost all palindromes are composite}, Mathematical Research
  Letters, 11 (2004), 853--868.
  
  \bibitem{Banks-Saidak-Sakata}{\sc W. D.~Banks, F.~Saidak  and M.~Sakata},
  {\em  Kloosterman sums for modified van der Corput sequences}, 
Unif. Distrib. Theory 2 (2007),  39--52. 

\bibitem{Banks-Shparlinski}{\sc W. D.~Banks and I.E.~ Shparlinski},
  {\em Prime divisors of palindromes}, Period. Math. Hungar., 51
  (2005), 1--10.

\bibitem{Bou1} {\sc J.~Bourgain}, {\em Prescribing the binary digits
    of primes}, Israel J. Math., 194 (2013), 935--955.
 
\bibitem{Bou2} {\sc J.~Bourgain}, {\em Prescribing the binary digits
    of primes, II}, Israel J. Math., 206 (2015), 165--182.
 
\bibitem{Bug1} {\sc Y. Bugeaud}, {\em On the digital representation of
    integers with bounded prime factors}, Osaka J. Math., 55 (2018),
  315--324.

\bibitem{BK} {\sc Y. Bugeaud and H. Kaneko}, {\em On the digital
    representation of smooth numbers}, Math. Proc. Cambridge
  Philos. Soc., 165 (2018), 533--540.

\bibitem{col-2009-ellipsephic}{\sc S.~Col}, {\em Diviseurs des nombres
    ellips\'ephiques}, Period. Math. Hungarica, 58 (2009), 1--23.

\bibitem{col-2009-palindromes} {\sc S.~Col}, {\em Palindromes dans les
    progressions arithm\'etiques}, Acta Arith., 137 (2009), 1--41.
    
\bibitem{DartygeMauduit2000} {\sc C. Dartyge and C. Mauduit}, {\em
    Nombres presque premiers dont l'\'ecriture en base $r$ ne comporte
    pas certains chiffres}, 
  J. Number Theory, 81 (2000), 270--291.

\bibitem{DartygeMauduit2001} {\sc C. Dartyge and C. Mauduit},
  {\em Ensembles de densit\'e nulle contenant des entiers
    poss\'edant au plus deux facteurs premiers}, 
    J. Number
  Theory,  91 (2001),   230--255.

\bibitem{DartygeTenenbaum2006}{\sc C. Dartyge and G. Tenenbaum} {\em
    Congruences de sommes de chiffres de valeurs polynomiales},
  Bull. London Math. Soc. 38 (2006), 61-69.

\bibitem{DES} {\sc R. Dietmann, C. Elsholtz and I. E.~Shparlinski},
  {\em Prescribing the binary digits of squarefree numbers and
    quadratic residues}, Trans. Amer. Math. Soc., 369 (2017),
  8369--8388.

\bibitem{diamond_halberstam_richert_1988} {\sc H.~Diamond,
    H.~Halberstam and H.-E.~Richert}, {\em Combinatorial sieves of
    dimension exceeding one}, J. Number Theory, 28 (1988), 306--346.

\bibitem{DMR2011} {\sc M. Drmota, C. Mauduit and J. Rivat}, {\em The
    sum of digits function of polynomial sequences}, J.  London
  Math. Soc., 84, (2011), 81--102.
   
\bibitem{DMR} {\sc M. Drmota, C. Mauduit and J. Rivat}, {\em Prime
    numbers in two bases}, Duke Math. J., 69 (2020), 1809--1876.

\bibitem{hardy-wright-1979} {\sc G.~H. Hardy and E.~M. Wright}, {\em
    {An Introduction to the Theory of Numbers}}, Oxford Science
  Publications, fifth~ed., 1979.

\bibitem{Irving}{\sc A. J.~Irving}, {\em Diophantine Approximation
    with Products of Two Primes}, J. London Math. Soc., 89 (2014),
  581--602.

\bibitem{Kar22} {\sc F. Karwatowski}, {\em Primes with one excluded
    digit}, Acta Arith., 202 (2022), 105--121.

\bibitem{MaRi1} {\sc C.~Mauduit and J.~Rivat}, {\em La somme des
    chiffres des carr\'{e}s}, Acta Math., 203 (2009), 107--148.

\bibitem{MaRi2} {\sc C.~Mauduit and J.~Rivat}, {\em Sur un probl\`eme
    de {G}elfond: la somme des chiffres des nombres premiers}, Ann. of
  Math.,171 (2010), 1591--1646.

\bibitem{May19} {\sc J. Maynard}, {\em Primes with restricted digits},
  Invent. Math., 217 (2019), 127--218.
 
\bibitem{May22} {\sc J. Maynard}, {\em Primes and polynomials with
    restricted digits}, Int. Math. Res. Not., 2022 (2022),
  10626--10648.
 
\bibitem{montgomery-1978} {\sc H. L.~Montgomery}, {\em The analytic
    principle of the large sieve}, Bull. Amer. Math. Soc., 84 (1978),
  547--567.

\bibitem{montgomery-vaughan-2007} {\sc H.~L. Montgomery and
    R.~C. Vaughan}, {\em Multiplicative number theory.
    {I}. {C}lassical theory}, Cambridge Studies in Adv.  Math. ,
  Cambridge Univ. Press, Cambridge, 2007.
   
\bibitem{Naslund} {\sc E. Naslund}, {\em The tail distribution of the
    sum of digits of prime numbers}, Unif. Distrib. Theor., 10 (2015),
  63--68.
  
\bibitem{OEIS} {\sc OEIS}, {\em On-line encyclopedia of integer
    sequences}, \url{https://oeis.org}.

\bibitem{Pratt} {\sc K.~Pratt}, {\em Primes from sums of two squares
    and missing digits}, Proc. Lond. Math. Soc., 20 (2020), 770--830.

\bibitem{Spiegelhofer} {\sc L.~Spiegelhofer}, {\em Thue–Morse along
    the sequence of cubes}, Preprint, 2023,
  \url{https://arxiv.org/abs/2308.09498}.
 
\bibitem{Stoll} {\sc T. Stoll}, {\em The sum of digits of polynomial
    values in arithmetic progressions}, Functiones et Approximatio 47
  (2012), 233-239.

\bibitem{Swa20}{\sc C.~Swaenepoel}, {\em Prime numbers with a positive
    proportion of preassigned digits}, Proc. London Math. Soc., 121
  (2020), 83--151.
  
\bibitem{Ten15} {\sc G.~Tenenbaum}, {\em Introduction to analytic and
    probabilistic number theory}, Grad. Studies in Math., vol.~163,
  AMS, 2015.

\bibitem{titchmarsh-1939} {\sc E. C.~Titchmarsh}, {\em The Theory of
    Functions}, Oxford University Press, London, second~ed., 1939.
  
\end{thebibliography}
\end{document}